\documentclass{amsart}

\usepackage[all]{xy}

\theoremstyle{plain}
\newtheorem{thm}{Theorem}[section]
\newtheorem{prop}[thm]{Proposition}
\newtheorem{lemma}[thm]{Lemma}

\newtheorem{question}[thm]{Question}

\theoremstyle{definition}
\newtheorem{dfn}[thm]{Definition}

\newtheorem{hypo}[thm]{Hypothesis}

\theoremstyle{remark}
\newtheorem{rem}[thm]{Remark}

 \newcommand{\unter}[2]{\genfrac{}{}{0pt}{}{#1}{#2}}

\begin{document}

\title{Universal deformation rings and generalized quaternion defect groups}

\author{Frauke M. Bleher}
\address{Department of Mathematics\\University of Iowa\\
Iowa City, IA 52242-1419}
\email{fbleher@math.uiowa.edu}
\thanks{The author was supported in part by  
NSA Grant H98230-06-1-0021 and NSF Grant DMS06-51332.}

\begin{abstract}
We determine the universal deformation rings $R(G,V)$ of certain mod $2$ representations $V$ of
a finite group $G$ 
which belong to a $2$-modular block of $G$ whose defect groups are isomorphic to a generalized
quaternion group $D$.
We show that for these $V$, a question raised by the author and Chinburg concerning 
the relation of $R(G,V)$ to  $D$ has an affirmative answer.  We also show that
$R(G,V)$ is a complete intersection even though $R(G/N,V)$ need not be for certain
normal subgroups $N$ of $G$ which act trivially on $V$.
\end{abstract}

\maketitle


\section{Introduction}
\label{s:intro}
\setcounter{equation}{0}
\setcounter{figure}{0}

Let $k$ be an algebraically closed field of positive characteristic $p$, let $W=W(k)$ be the ring
of infinite Witt vectors over $k$, and let $G$ be a finite group.
It is a classical question to ask whether a given finitely generated $kG$-module $V$ can be lifted
to $W$ or to a more general complete local commutative Noetherian ring $R$ with residue field $k$.
It is useful to formulate this question in terms of deformation rings.  For example, Green's liftability
theorem can be stated as saying that there is a 
surjection from the versal deformation ring $R(G,V)$ of $V$ to $W$ 
if there are no non-trivial $2$-extensions of $V$ by itself. 
A natural question is then to determine $R(G,V)$.
In this paper, we determine $R(G,V)$ for certain mod $2$ representations $V$ of finite groups $G$ 
which belong to $2$-modular blocks of $G$ with generalized quaternion defect groups.

The topological ring $R(G,V)$ is characterized by the property that every lift of $V$ over
a complete local commutative Noetherian ring $R$ with residue field $k$ arises
from a local ring homomorphism $\alpha: R(G,V)\to R$ and that $\alpha$ is unique
if $R$ is the ring of dual numbers $k[\epsilon]/(\epsilon^2)$. In case $\alpha$ is unique
for all $R$, $R(G,V)$ is called the universal deformation ring of $V$.
For precise definitions, see \S\ref{s:prelim}. 
Note that all these rings $R$, including $R(G,V)$, have a natural structure as $W$-algebras.

One of the main motivations for studying universal deformation rings for finite groups is that they
provide a good test case for various conjectures concerning the ring theoretic properties
of universal deformation rings for profinite Galois groups. 
For example, in \cite{bc4.9,bc5} it was shown that the universal deformation ring of the non-trivial
irreducible mod $2$ representation of the symmetric group $S_4$ is not a complete intersection.
This led to infinitely many examples of real quadratic fields $L$ such that the universal deformation 
ring of the inflation of this representation to the  Galois group over $L$ of the 
maximal totally unramified extension of $L$ is not a complete intersection.
The advantage of computing universal deformation rings for representations of finite groups
is that one can use deep results from modular representation theory due to Brauer, 
Erdmann \cite{erd}, Linckelmann \cite{linckel,linckel1}, Carlson-Th\'{e}venaz \cite{carl2,carl1.5}, and 
others. 

Suppose now that $G$ is an arbitrary finite group and $V$ is a 
finitely generated $kG$-module such that the stable 
endomorphism ring $\underline{\mathrm{End}}_{kG}(V)$ is isomorphic to $k$. 
By \cite[Prop. 2.1]{bc}, it follows that 
the versal deformation ring $R(G,V)$ of $V$ is universal. In \cite{bc}, the following question was
raised which would relate $R(G,V)$ to the defect groups
of the block of $kG$ associated to $V$.

\begin{question}
\label{qu:main} 
Let $B$ be a block of $kG$, let $D$ be a defect group of $B$, and suppose $V$ is a finitely generated 
$kG$-module whose stable endomorphism ring is isomorphic to $k$ 
such that the unique $($up to isomorphism$)$ non-projective indecomposable summand of $V$ 
belongs to $B$. Is the universal deformation ring $R(G,V)$ of $V$ isomorphic to a subquotient ring 
of the group ring $WD$?
\end{question}

In \cite{bl,diloc,3sim,bc,bllo}, the isomorphism types of $R(G,V)$ have been determined for $V$ 
belonging to cyclic blocks, respectively to various tame blocks with dihedral defect groups. 
It was shown that Question \ref{qu:main} has a positive answer in all these cases.
Moreover, in \cite[Cor. 5.1.2]{3sim} an infinite series of  finite  groups $G$ and mod $2$ 
representations $V$  was given for which $R(G,V)$ is not a complete intersection.

In this paper, we consider the case when $k$ has characteristic $2$ and 
$B$ is a block of $kG$ with generalized quaternion defect groups of order $2^{d+1}\ge 8$
such that the center of $G/O_B$ has even order, where $O_B$ is the maximal
normal subgroup of $G$ of odd order which acts trivially on all $kG$-modules belonging to $B$.
Let $z\in G$ be an involution such that $zO_B$ lies in the center of $G/O_B$, let $N=\langle O_B,z\rangle$
and let $\overline{B}$ be a block of $k[G/N]$ which is contained in the image of $B$ under the 
natural surjection $\pi:kG\to k[G/N]$. 
Then $\overline{B}$ has dihedral defect groups of order $2^d\ge 4$, and 
$B$ and $\overline{B}$ have the same number of isomorphism classes of simple modules
(see Remark \ref{rem:trythis}). 
In \cite{brauer2} (resp. \cite{olsson}), Brauer (resp. Olsson) proved that a block with 
dihedral (resp. generalized quaternion) defect groups contains at most three simple modules up to 
isomorphism. In this paper, we consider the largest case when both $B$ and $\overline{B}$ have
precisely three isomorphism classes of simple modules.

Note that all principal blocks
with generalized quaternion defect groups and precisely three isomorphism classes of
simple modules are included in our discussion (see Remark \ref{rem:principal}).
If $B$ is principal and $d\ge 3$, then $\overline{B}$ is one of the blocks considered in \cite{3sim} 
(see Remark \ref{rem:oherdmann}).

We now summarize our main result;  a more detailed statement can be found in 
Theorem \ref{thm:altogether}. As above, if $B$ is a block of $kG$ then $O_B$ is the maximal 
normal subgroup of $G$ of odd order which acts trivially on all $kG$-modules belonging to $B$.
\begin{thm}
\label{thm:main}
Suppose $k$ has characteristic $2$, $G$ is a finite group, and $B$ is a block of $kG$ with a
generalized quaternion defect group $D$ of order $2^{d+1}\ge 8$ such that 
there are precisely three isomorphism classes of simple $B$-modules and such that
the center of $G/O_B$ has even order.
Let $z\in G$ be an involution such that $zO_B$ is central in $G/O_B$, let $N=\langle O_B,z\rangle$ 
and let $\overline{B}$ be a block of 
$k[G/N]$ which is contained in the image of $B$ under the natural surjection $\pi:kG\to k[G/N]$.
Let $V$ be a finitely generated $kG$-module whose stable endomorphism ring is isomorphic to $k$ 
and which is inflated from a $k[G/N]$-module belonging to $\overline{B}$.
Then either
\begin{enumerate}
\item[(i)] $R(G,V)/2R(G,V)\cong k$, in which case $R(G,V)$ is isomorphic to $W$, or
\item[(ii)] $R(G,V)/2R(G,V)\cong k[[t]]/(t^{2^{d-1}-1})$, in which case  $R(G,V)\cong W[[t]]/(p_{d+1}(t))$ 
for a certain monic polynomial $p_{d+1}(t)\in W[t]$ of degree $2^{d-1}-1$ whose non-leading 
coefficients are all divisible by $2$.
\end{enumerate}
In all cases, $R(G,V)$ is isomorphic to a subquotient ring of $WD$ and a complete intersection.
\end{thm}

It is an interesting question how the universal deformation ring changes
when one inflates a module $V$ from a quotient group of $G$  to $G$. 
Theorem \ref{thm:main} together with the results in \cite{bl,3sim}
and Lemma \ref{lem:moreIII}
give an answer to this question for the quotient group $G/N$.
Namely, using the notation of Theorem \ref{thm:main}, 
the universal deformation ring $R(G/N,V)$ is as follows.
\begin{itemize}
\item If $V$ is as in Theorem \ref{thm:main}(i), then $R(G/N,V)$ is isomorphic to a quotient ring
of $W$.
\item If $V$ is as in Theorem \ref{thm:main}(ii), then $R(G/N,V)\cong W[[t]]/(t\,p_d(t),2\,p_d(t))$,
where we set $p_2(t)=1$. In particular, if $d\ge 3$, $R(G/N,V)$ is not a complete 
intersection.
\end{itemize}

\medbreak

It is natural  to ask whether Theorem \ref{thm:main} can be used to construct 
deformation rings arising from arithmetic in the following sense.
Suppose $L$ is a number field and $S$ is a finite set of places of $L$. Let $L_S$ be
the maximal algebraic extension of $L$ which is unramified outside $S$, and denote
by $G_{L,S}$ the Galois group of $L_S$ over $L$.  Suppose $k$, $G$, $G/N$ and $V$ are
as in Theorem \ref{thm:main}, and let $H$ be $G$ or $G/N$.
As in \cite{bc5}, one can ask whether there are $L$ and $S$ such 
that there is a surjection $\psi:G_{L,S}\to H$ which induces an isomorphism 
of deformation rings $R(G_{L,S},V)\to R(H, V)$ 
when $V$ is viewed as a representation for $G_{L,S}$ via $\psi$. It was shown in 
\cite[Lemma 3.3]{bc5} that a sufficient condition for $R(G_{L,S},V)\to R(H, V)$ to be an isomorphism 
for all such $V$ is that  $\mathrm{Ker}(\psi)$ has no non-trivial pro-$2$ quotient. If this condition is
satisfied, we say the group $H$ caps $L$ for $2$ at $S$. 

The prototypes for the groups $G/N$ are $\mathrm{PSL}_2(\mathbb{F}_q)$ where $q$ is an
odd prime power and the alternating group $A_7$. Suppose $G/N$ is one of these groups  and that
$8$ divides the order of $G/N$. One can show similarly to the proof of \cite[Thm. 3.7(i)]{bc5}
that $G/N$ does not cap $\mathbb{Q}$ for $2$ at any finite set
of places $S$ of $\mathbb{Q}$. However, it is an interesting question whether the group
$G$ does cap $\mathbb{Q}$ for $2$ at certain $S$. For example, it was shown in
\cite{bcf} that although the symmetric group $S_4$ does not cap $\mathbb{Q}$ for $2$
at any $S$, the double cover $\hat{S}_4$ does cap $\mathbb{Q}$ for $2$ at certain $S$.
This situation was similar to the present one, since the universal deformation ring of 
the simple $kS_4$-module $E$ of dimension $2$ is $R(S_4,E)\cong W[[t]]/(t^2,2t)=
W[[t]]/(t\,p_3(t),2\,p_3(t))$, but the universal deformation ring of the inflation of $E$ to $\hat{S}_4$ 
is $R(\hat{S}_4,E)\cong W[[t]]/(t^3-2t)=W[[t]]/(p_4(t))$.

The paper is organized as follows. In \S\ref{s:prelim}, we recall the definitions of deformations and 
deformation rings and we state some useful results from \cite{3sim,bc}.
In \S\ref{s:quaternionsylow}, we concentrate on 
finite groups $G$ and their quotient groups $G/N$ as in Theorem \ref{thm:main} 
and use  \cite{erd} to describe the $2$-modular blocks $B$ and $\overline{B}$.
In \S\ref{s:stablend}, we determine all indecomposable $kG$-modules $V$ whose stable
endomorphism rings are isomorphic to $k$ and which are inflated from $k[G/N]$-modules belonging to
$\overline{B}$. In particular, we show that the endomorphism ring of each such $V$ is isomorphic to $k$,
leading to a finite set of possibilities for such $V$ (see Remark \ref{rem:stablend}). 
In \S\ref{s:ordinaryquaternion}, we describe results of Olsson \cite{olsson} about the ordinary 
irreducible characters of $G$ belonging to $B$. 
In \S\ref{s:udr}, we determine the universal deformation rings of all the $kG$-modules $V$ found in 
\S\ref{s:stablend}, by first determining the universal deformation rings modulo $2$ and then using 
Olsson's results on ordinary characters to complete the computation
(see Theorem \ref{thm:altogether}).


\section{Preliminaries}
\label{s:prelim}
\setcounter{equation}{0}
\setcounter{figure}{0}

In this section, we give a brief introduction to versal and universal deformation rings and deformations. 
For more background material, we refer the reader to \cite{maz1} and \cite{lendesmit}.

Let $k$ be an algebraically closed field of characteristic $p>0$, let $W$ be the ring of infinite Witt 
vectors over $k$, and let $F$ be the fraction field of $W$. 
Let $\hat{\mathcal{C}}$ be the category of all complete local commutative Noetherian 
rings with residue field $k$. The morphisms in $\hat{\mathcal{C}}$ are continuous $W$-algebra 
homomorphisms which induce the identity map on $k$.

Suppose $G$ is a finite group and $V$ is a finitely generated $kG$-module. 
A lift of $V$ over an object $R$ in $\hat{\mathcal{C}}$ is a pair $(M,\phi)$ where $M$ is a finitely 
generated $RG$-module which is free over $R$, and $\phi:k\otimes_R M\to V$ is an isomorphism of 
$kG$-modules. Two lifts $(M,\phi)$ and $(M',\phi')$ of $V$ over $R$ are isomorphic if there is an 
isomorphism $f:M\to M'$ with $\phi=\phi'\circ (k\otimes f)$. The isomorphism class $[M,\phi]$ of a lift 
$(M,\phi)$ of $V$ over $R$ is called a deformation of $V$ over $R$, and the set of all such deformations 
is denoted by $\mathrm{Def}_G(V,R)$. The deformation functor
$$\hat{F}_V:\hat{\mathcal{C}} \to \mathrm{Sets}$$ 
is a covariant functor which
sends an object $R$ in $\hat{\mathcal{C}}$ to $\mathrm{Def}_G(V,R)$ and a morphism 
$\alpha:R\to R'$ in $\hat{\mathcal{C}}$ to the map $\mathrm{Def}_G(V,R) \to
\mathrm{Def}_G(V,R')$ defined by $[M,\phi]\mapsto [R'\otimes_{R,\alpha} M,\phi_\alpha]$, where  
$\phi_\alpha=\phi$ after identifying $k\otimes_{R'}(R'\otimes_{R,\alpha} M)$ with $k\otimes_R M$.

Suppose there exists an object $R(G,V)$ in $\hat{\mathcal{C}}$ and a deformation 
$[U(G,V),\phi_U]$ of $V$ over $R(G,V)$ with the following property:
For each $R$ in $\hat{\mathcal{C}}$ and for each lift $(M,\phi)$ of $V$ over $R$ there exists 
a morphism $\alpha:R(G,V)\to R$ in $\hat{\mathcal{C}}$ such that $\hat{F}_V(\alpha)([U(G,V),\phi_U])=
[M,\phi]$, and moreover $\alpha$ is unique if $R$ is the ring of dual numbers
$k[\epsilon]/(\epsilon^2)$. Then $R(G,V)$ is called the versal deformation ring of $V$ and 
$[U(G,V),\phi_U]$ is called the versal deformation of $V$. If the morphism $\alpha$ is
unique for all $R$ and all lifts $(M,\phi)$ of $V$ over $R$, 
then $R(G,V)$ is called the universal deformation ring of $V$ and $[U(G,V),\phi_U]$ is 
called the universal deformation of $V$. In other words, $R(G,V)$ is universal if and only if
$R(G,V)$ represents the functor $\hat{F}_V$ in the sense that $\hat{F}_V$ is naturally isomorphic to 
the Hom functor
$\mathrm{Hom}_{\hat{\mathcal{C}}}(R(G,V),-)$. 

By \cite{maz1}, every finitely generated $kG$-module $V$ has a versal deformation ring.
By a result of Faltings (see \cite[Prop. 7.1]{lendesmit}), $V$ has a universal deformation ring if 
$\mathrm{End}_{kG}(V)=k$.

\begin{rem}
\label{rem:deformations}
Note that the above definition of deformations differs from the definition used in \cite{bl,bc} as follows.
Let $G$ and $V$ be as above, let $R$ be an object in $\hat{\mathcal{C}}$ and let $(M,\phi)$
be a lift of $V$ over $R$. In \cite{bl,bc}, the isomorphism class of $M$ as an $RG$-module was called a 
deformation of $V$ over $R$, without taking into account the specific isomorphism 
$\phi:k\otimes_RM\to V$. Some authors call a deformation according to the latter definition 
a weak deformation of $V$ over $R$ (see \cite[\S 5.2]{keller}). 

In general, a weak deformation of $V$ over $R$ identifies more lifts than a deformation of $V$ over $R$
that respects the isomorphism $\phi$ of a representative $(M,\phi)$.
However, if the stable
endomorphism ring $\underline{\mathrm{End}}_{kG}(V)$ is isomorphic to $k$, these two definitions
of deformations coincide. This can be seen as follows. 

Suppose first that $\alpha:A\to A_0$ is a surjective
morphism of Artinian rings $A,A_0$ in $\hat{\mathcal{C}}$ and $(M,\phi)$ is a lift of $V$ over
$A$. Let $(M_0,\phi_0)=(A_0\otimes_{A,\alpha}M,\phi_\alpha)$ and let $u_0\in\mathrm{End}_{A_0G}(M_0)$ be
such that $u_0$ factors through a projective $A_0G$-module $P_0$, say $u_0=h_0\circ g_0$. 
We claim that $u_0$ can be lifted to an $AG$-module endomorphism $u$ of $M$ that factors 
through a projective $AG$-module $P$.
By using induction on the length of $A$, it is enough to prove this when $\alpha$ is a small extension, i.e.
when the kernel of $\alpha$ is a non-zero
principal ideal $tA$ where $t$ is annihilated by the maximal ideal $m_A$ of $A$. 
Since $M_0$ is a finitely generated $A_0G$-module, we can assume that $P_0$ is a finitely 
generated projective $A_0G$-module. Let $P$ be a projective $AG$-module with $A_0\otimes_{A,\alpha}
P\cong P_0$. 
Consider the short exact sequence of $AG$-modules $0\to tP\to P\xrightarrow{\alpha_P} P_0\to 0$ which 
results from tensoring $0\to tA \to A \xrightarrow{\alpha} A_0\to 0$ with $P$ over $A$.
Since $k\otimes_A P$ is an injective $kG$-module, we have that
$\mathrm{Ext}^1_{AG}(M,tP)\cong \mathrm{Ext}^1_{kG}(k\otimes_AM,k\otimes_A P)= 0$. Therefore,
$\mathrm{Hom}_{AG}(M,P)\xrightarrow{(\alpha_P)_*}\mathrm{Hom}_{AG}(M,P_0)$ is surjective.
This implies that $g_0:M_0\to P_0$ can be
lifted to an $AG$-module homomorphism $g:M\to P$. Since $P$ is a projective $AG$-module, it
follows that $h_0:P_0\to M_0$ can be lifted to an $AG$-module homomorphism $h:P\to M$.
Hence $u=h\circ g$ is an $AG$-module endomorphism of $M$ which factors through $P$ and which
lifts $u_0$. Since by \cite[Lemma 2.3]{bc} the ring homomorphism 
$\sigma_M:A\to \underline{\mathrm{End}}_{AG}(M)$ coming from the action of
$A$ on $M$ via scalar multiplication is surjective, it follows that every endomorphism 
$f_0\in\mathrm{End}_{A_0G}(M_0)$ can be lifted to an endomorphism $f\in\mathrm{End}_{AG}(M)$.

Now suppose that $R$ is an arbitrary ring in $\hat{\mathcal{C}}$, i.e.
$R=\displaystyle\lim_{\longleftarrow} R/m_R^n$ where $m_R$ is the
maximal ideal of $R$. If $(M,\phi)$ and $(M',\phi')$ are two lifts of $V$ over $R$, then
$((R/m_R^n)\otimes_RM,\phi_n)$ and $((R/m_R^n)\otimes_RM',\phi'_n)$ are lifts of $V$ over
$R/m_R^n$ for all $n$, where $\phi_n=\phi$ (resp. $\phi'_n=\phi'$)
after identifying $k\otimes_{R/m_R^n}((R/m_R^n)\otimes_R M)$ with $k\otimes_R M$ (resp.
$k\otimes_{R/m_R^n}((R/m_R^n)\otimes_R M')$ with $k\otimes_R M'$). 
Suppose there exists an $RG$-module isomorphism $f:M\to M'$. Then 
$f_n=(R/m_R^n)\otimes_Rf:(R/m_R^n)\otimes_RM\to (R/m_R^n)\otimes_RM'$ is an 
$(R/m_R^n)G$-module isomorphism for all $n$.
Define $g_1=\phi^{-1}\circ\phi'\circ f_1$, which is a $kG$-module automorphism of $k\otimes_RM$.
By what we showed above, we can inductively lift $g_1$ to an $(R/m_R^n)G$-module endomorphism
$g_n$ of $(R/m_R^n)\otimes_R M$ such that $g_{n+1}$ lifts $g_n$ for all $n$. Since $k\otimes g_n
=g_1$ is an automorphism of $k\otimes_R M$ and $(R/m_R^n)\otimes_R M$ is free over $R/m_R^n$, 
it follows that
$g_n$ is an $(R/m_R^n)G$-module automorphism of $(R/m_R^n)\otimes_R M$. For all $n\ge 1$, define
$h_n=f_n\circ g_n^{-1}$. Then $h_n: (R/m_R^n)\otimes_RM\to (R/m_R^n)\otimes_RM'$ is an 
$(R/m_R^n)G$-module isomorphism such that $\phi'_n\circ (k\otimes h_n)= \phi_n$
and such that $h_{n+1}$ lifts $h_n$ for all $n$. Hence $\{h_n\}$
defines an inverse system and its inverse limit defines an $RG$-module isomorphism
$h:M\to M'$ with $\phi'\circ (k\otimes h)=\phi$. In particular, $[M,\phi]=[M',\phi']$.
\end{rem}

The following three results were proved in \cite{bc} and in \cite{3sim}, respectively. 
Here $\Omega$ denotes
the syzygy, or Heller, operator for $kG$ (see for example \cite[\S 20]{alp}).

\begin{prop}
\label{prop:stablendudr}
{\rm (\cite[Prop. 2.1]{bc}).}
Suppose $V$ is a finitely generated $kG$-module whose stable endomorphism ring 
$\underline{\mathrm{End}}_{kG}(V)$ is isomorphic to $k$.  Then $V$ has  a universal 
deformation ring $R(G,V)$.
\end{prop}

\begin{lemma} 
\label{lem:defhelp}
{\rm (\cite[Cors. 2.5 and 2.8]{bc}).}
Let $V$ be a finitely generated $kG$-module with $\underline{\mathrm{End}}_{kG}(V)\cong k$.
\begin{enumerate}
\item[(i)] Then $\underline{\mathrm{End}}_{kG}(\Omega(V))\cong k$, and $R(G,V)$ and $R(G,\Omega(V))$ 
are isomorphic.
\item[(ii)] There is a non-projective indecomposable $kG$-module $V_0$ $($unique up to
isomorphism$)$ such that $\underline{\mathrm{End}}_{kG}(V_0)\cong k$, $V$ is isomorphic to 
$V_0\oplus P$ for some projective $kG$-module $P$, and $R(G,V)$ and $R(G,V_0)$ are 
isomorphic.
\end{enumerate}
\end{lemma}

\begin{lemma}
\label{lem:Wlift}
{\rm (\cite[Lemma 2.3.2]{3sim}).}
Let $V$ be a finitely generated $kG$-module such that there is a non-split short exact sequence of 
$kG$-modules
$$0\to Y_2\to V\to Y_1\to 0$$
with $\mathrm{Ext}^1_{kG}(Y_1,Y_2)\cong k$.
Suppose that for $i\in\{1,2\}$, there exists a $WG$-module $X_i$ which defines a lift of $Y_i$ over $W$. 
Suppose further that 
$$\mathrm{dim}_F\;\mathrm{Hom}_{FG}(F\otimes_WX_1,F\otimes_WX_2) =
\mathrm{dim}_k\;\mathrm{Hom}_{kG}(Y_1,Y_2)-1.$$
Then there exists a $WG$-module $X$ which defines a lift of $V$ over $W$.
\end{lemma}

We also need the following result which is a more general version of \cite[Lemma 2.3.1]{3sim}.

\begin{lemma}
\label{lem:trythis}
Let $Y$ be a finitely generated uniserial $kG$-module 
satisfying $\mathrm{End}_{kG}(Y)\cong k$ 
and $\mathrm{Ext}^1_{kG}(Y,Y)\cong k$. Suppose that the radical length of $Y$ is $\ell\ge 1$
and that the composition factors in the descending radical series are given as 
$(T_1,T_2,\ldots, T_\ell)$ where $T_1,\ldots,T_\ell$ are simple $kG$-modules, not necessarily distinct.
Assume there exists an integer $r\ge 2$ such that the projective cover $P_{T_1}$ of $T_1$
has a quotient module $\overline{U}$ which is a uniserial $kG$-module of radical length
$\ell r$ such that the composition factors in the descending radical series are given as
$$(T_1,T_2,\ldots,T_\ell,T_1,T_2,\ldots,T_\ell,\ldots,T_1,T_2,\ldots,T_\ell).$$
Suppose there are $kG$-module isomorphisms  
$\phi: \overline{U} /\mathrm{rad}^{\ell} (\overline{U})\to  Y$ and 
$\psi:\overline{U} /\mathrm{rad}^{\ell (r-1)} (\overline{U})\to
\mathrm{rad}^{\ell}(\overline{U})$, and assume that
$\mathrm{Ext}^1_{kG}(\overline{U},Y)=0$. Then 
the universal deformation ring of $Y$ 
modulo $p$ is $\overline{R}=R(G,Y)/pR(G,Y)\cong k[t]/(t^{r})$, and the universal mod $p$ deformation 
of $Y$ over $\overline{R}$ is 
$[\overline{U},\phi]$
where the action of $t$ on $\overline{U}$ is induced by $\psi$.
\end{lemma}

\begin{proof}
By assumption, $\mathrm{Ext}^1_{kG}(Y,Y)\cong k$, which implies that $\overline{R}\cong k[[t]]$ or
$\overline{R}\cong k[t]/(t^m)$ for some $m\ge 2$. 
The action of $t$ on $\overline{U}$ is given by the composition
$$\overline{U} \xrightarrow{\pi_U} \overline{U}/\mathrm{rad}^{\ell(r-1)} (\overline{U}) 
\xrightarrow{\psi} \mathrm{rad}^{\ell} (\overline{U})\xrightarrow{\iota_U}\overline{U}$$
where $\pi_U$ (resp. $\iota_U$) is the natural surjection (resp. inclusion).
It follows that $\mathrm{rad}^{\ell} (\overline{U})=t \,\overline{U}$ and that
$\overline{U}$ is a free $k[t]/(t^r)$-module under this action. In particular, $\phi$ defines a 
$kG$-module isomorphism $k\otimes_{k[t]/(t^r)}\overline{U}=\overline{U} /\mathrm{rad}^{\ell} 
(\overline{U})\to  Y$, and hence $(\overline{U},\phi)$ defines a lift of $Y$ over $k[t]/(t^r)$.
This means that there is a unique $k$-algebra homomorphism
$$\alpha:\overline{R}\to k[t]/(t^{r})$$
corresponding to the deformation $[\overline{U},\phi]$. Since $\overline{U}$ is indecomposable as a 
$kG$-module, it follows that $\alpha$ is surjective. We now show that $\alpha$ is a $k$-algebra 
isomorphism. Suppose this is false. Then there exists a surjective $k$-algebra homomorphism 
$\alpha_1:\overline{R}\to k[t]/(t^{r+1})$ such that $\tau\circ\alpha_1=\alpha$ where 
$\tau:k[t]/(t^{r+1})\to k[t]/(t^{r})$ is the natural projection. Let $(\overline{U}_1,\phi_1)$ be a lift of $Y$ over 
$k[t]/(t^{r+1})$ corresponding to $\alpha_1$. Then $k[t]/(t^r)\otimes_{k[t]/(t^{r+1}),\tau}\overline{U}_1
\cong \overline{U}$ and $t^{r}\overline{U}_1\cong Y$. We have a short 
exact sequence of $k[t]/(t^{r+1})\,G$-modules
\begin{equation}
\label{eq:thesos}
0\to t^{r}\overline{U}_1\to \overline{U}_1\to \overline{U}\to 0.
\end{equation}
We now show that this sequence cannot split as a sequence of $kG$-modules. Suppose it splits. Then 
$\overline{U}_1\cong Y\oplus \overline{U}$ as $kG$-modules. Let 
$z=\left(\begin{array}{c}y\\u\end{array}\right)\in Y\oplus \overline{U} \cong \overline{U}_1$. 
Then $t$ acts on $z$ as multiplication by a matrix of the form
$$M_t=\left(\begin{array}{cc}0&\gamma\\0&\mu_t\end{array}\right)$$ 
where $\gamma:\overline{U}\to Y$ is a surjective $kG$-module homomorphism, and $\mu_t$ is 
multiplication by $t$ on $\overline{U}$. Since $t^{r}\overline{U}_1\cong Y$, 
$(M_t)^r$ cannot be the zero matrix.
Because $\mathrm{End}_{kG}(Y)\cong k$, $\gamma$ corresponds, up to a non-zero
scalar multiple,  to the isomorphism $\phi:\overline{U}/t\,\overline{U}\to Y$, which means that the kernel of 
$\gamma$ is $t\,\overline{U}$. 
This implies that $(M_t)^r$ is the zero matrix, 
which is a contradiction.
Hence the short exact sequence (\ref{eq:thesos}) does not split as a sequence of $kG$-modules. 
Since $\mathrm{Ext}^1_{kG}(\overline{U},Y)=0$ by assumption and $t^{r}\overline{U}_1\cong Y$, this 
is impossible. Therefore, $\overline{U}_1$ does not exist, which means that $\alpha$ is a 
$k$-algebra isomorphism. Thus $\overline{R}\cong k[t]/(t^{r})$, and the universal mod $p$ 
deformation of $Y$ over $\overline{R}$ is 
$[\overline{U},\phi]$.
\end{proof}


\section{Blocks with generalized quaternion defect groups}
\label{s:quaternionsylow}
\setcounter{equation}{0}
\setcounter{figure}{0}

For the remainder of this paper, we make the following assumptions.

\begin{hypo}
\label{hyp:alltheway}
Let $k$ be an algebraically closed field of characteristic $2$. Suppose $G$ is a finite group 
and $B$ is a block of $kG$ with a generalized quaternion defect group $D$ of order $2^{d+1}\ge 8$
such that there are precisely three isomorphism classes of simple $B$-modules.
Assume further that the center of $G/O_B$ has even order, where $O_B$ is the maximal
normal subgroup of $G$ of odd order which acts trivially on all $kG$-modules belonging to $B$.
Let $z\in G$ be an involution such that $zO_B$ lies in the center of $G/O_B$, and let
$N=\langle O_B,z\rangle$. Let $\overline{B}$ be a 
block of $k[G/N]$ which is contained in the image of $B$ under the natural surjection $\pi:kG\to k[G/N]$. 
\end{hypo}

Note that Olsson proved in \cite{olsson} that a block with generalized quaternion defect groups contains at 
most three simple modules up to isomorphism. Hence we consider the largest case.

\begin{rem}
\label{rem:principal}
Suppose $k$ is an algebraically closed field of characteristic $2$ and $G$ is a finite group with generalized 
quaternion Sylow $2$-subgroups of order $2^{d+1}\ge 8$. If $B$ is the principal block of $kG$, then
$O_B=O_{2'}(G)$, where $O_{2'}(G)$ is the maximal normal subgroup of $G$ of odd order.
By a result of Brauer and Suzuki \cite{brsu,br1}, the center of $G/O_{2'}(G)$ has order $2$. 
Hence all assumptions in Hypothesis
\ref{hyp:alltheway} are satisfied provided there are precisely three isomorphism classes of 
simple $B$-modules.
\end{rem}

The following remark discusses the basic properties of $B$ and $\overline{B}$ as in Hypothesis 
\ref{hyp:alltheway}.

\begin{rem}
\label{rem:trythis}
Assume Hypothesis $\ref{hyp:alltheway}$.
Since $O_B$ has odd order and $\mathrm{char}(k)=2$, the blocks of $k[G/O_B]$
correspond to the blocks of $kG$ whose primitive central
idempotents occur in the decomposition of the central idempotent $\frac{1}{\#O_B}
\sum_{g\in O_B}g$. Because $O_B$ acts trivially on all $kG$-modules belonging to $B$,
it follows that $B$ can be identified with a block $B_{O_B}$ of $k[G/O_B]$ and 
all $kG$-modules $V,V'$ belonging to $B$ can be identified with $k[G/O_B]$-modules belonging 
to $B_{O_B}$. We obtain that $\mathrm{Hom}_{kG}(V,V') = \mathrm{Hom}_{k[G/O_B]}(V,V')$,
$\underline{\mathrm{Hom}}_{kG}(V,V') =\underline{ \mathrm{Hom}}_{k[G/O_B]}(V,V')$
and $\mathrm{Ext}^i_{kG}(V,V')=\mathrm{Ext}^i_{k[G/O_B]}(V,V')$ for all $i\ge 1$.

Let $S$ and $S'$ be simple $kG$-modules belonging to $B$. Since their restriction to $N$ is trivial,
they can also be viewed as $k[G/N]$-modules. The Lyndon-Hochschild-Serre spectral sequence 
gives an exact sequence
$$0\to \mathrm{Ext}^1_{k[G/N]}(S,S')\to \mathrm{Ext}^1_{k[G/O_B]}(S,S') \to \mathrm{Hom}_{k[G/N]}
(S,S').$$
This implies immediately that $\mathrm{Ext}^1_{kG}(S,S')$ and $\mathrm{Ext}^1_{k[G/N]}(S,S')$ have 
the same $k$-dimension for each choice of non-isomorphic $S$ and $S'$. Since $\overline{B}$ is contained
in the image of $B$ under the natural surjection  $\pi:kG\to k[G/N]$, it follows by \cite[Prop. 13.3]{alp}
that $\overline{B}$ has the same number of isomorphism classes of simple modules as $B$.
Hence we can identify the simple $B$-modules with the simple $\overline{B}$-modules. 
It follows that the restriction of $\pi$ to $B$ gives a surjective $k$-algebra
homomorphism $\pi_B:B\to \overline{B}$. In particular, this implies that if $V$ is a $k[G/N]$-module
belonging to $\overline{B}$, then its inflation to $kG$ via $\pi$ belongs to $B$.

Define $D_{O_B}=DO_B/O_B$ and $\overline{D}=DN/N$. 
Then $D_{O_B}$ is a defect group of $B_{O_B}$. By \cite[Thm. 13.6]{alp}, $zO_B$ lies in $D_{O_B}$,
which implies $\overline{D}\cong D_{O_B}/\langle z O_B\rangle$.
Since the simple $B$-modules can be identified with the simple $\overline{B}$-modules, it follows 
by \cite[Prop. (56.32)]{CR} that the order of the defect groups of $\overline{B}$ is 
$2^d=\# \overline{D}$. Because all indecomposable $B$-modules are $(G,D)$-projective, it follows that all
indecomposable $\overline{B}$-modules are $(G/N,\overline{D})$-projective. Since there
exists an indecomposable $\overline{B}$-module which has a defect group of $\overline{B}$ as a vertex
(see \cite[Thm. (59.10)]{CR}), it follows that $\overline{D}$ is a defect group of $\overline{B}$.
In particular, this implies that the defect groups of $\overline{B}$ are dihedral groups of order $2^d$.
\end{rem}

\begin{lemma}
\label{lem:projectives}
Assume Hypothesis $\ref{hyp:alltheway}$, and
let $\pi_B:B\to\overline{B}$ be the surjective $k$-algebra homomorphism from Remark $\ref{rem:trythis}$. 
Suppose $\overline{P}$ is a projective indecomposable $\overline{B}$-module,
and denote its inflation to $B$ via $\pi_B$ also by $\overline{P}$. Let $P$ be a projective 
indecomposable $B$-module which is a projective cover of $\overline{P}$. Then 
there is a short exact sequence of $B$-modules
$$0\to \overline{P}\to P\to \overline{P}\to 0.$$
\end{lemma}

\begin{proof}
By Remark \ref{rem:trythis}, we can assume without loss
of generality that $O_B$ is trivial, which means that the involution $z$ lies in the center of $G$
and $N=\langle z \rangle$. Hence the kernel of the natural projection $\pi:kG\to k[G/N]$ 
is $\mathrm{Ker}(\pi)=(1+z)kG$, which implies that the kernel of $\pi_B$ is 
$\mathrm{Ker}(\pi_B)=(1+z)B$. Because $P$ is a projective $B$-module, we 
obtain a short exact sequence of $B$-modules
$$0\to (1+z)P \to P \to \overline{P}\to 0.$$
Since the map $f:(1+z)P\to \overline{B}\otimes_B P$, defined by $f((1+z)x)=1\otimes x$ for 
all $x\in P$, 
is a $B$-module isomorphism, Lemma \ref{lem:projectives} follows.
\end{proof}

Assume Hypothesis \ref{hyp:alltheway}.
From Erdmann's classification of all blocks of tame representation type 
in \cite{erd}, and by using Lemma \ref{lem:projectives}, it follows that the quivers of the basic algebras 
of $B$ and $\overline{B}$ can be identified and that, up to Morita equivalence, there
are precisely three families (I), (II) and (III) of blocks $\overline{B}$ and $B$.

Using \cite{erd}, we now give a  description of these families as follows.
For each of the three families, we give a quiver $Q$ and ideals $I$ and $\overline{I}$
of $kQ$ such that
$B$ is Morita equivalent to $\Lambda=kQ/I$ and $\overline{B}$ is Morita 
equivalent to $\overline{\Lambda}=kQ/\overline{I}$.
We denote the simple $\Lambda$-modules by $S_0,S_1,S_2$, or, using short-hand, by $0,1,2$.
The corresponding projective indecomposable $\Lambda$-modules are denoted by
$P_0$, $P_1$ and $P_2$. We describe the radical series of $P_0$, $P_1$ and $P_2$
and also provide the decomposition matrix of $B$. 
All figures appear at the end of the paper.

\subsection{Family (I)}
\label{ss:psl1A}
For Family (I), $d\ge 2$. The quiver $Q$ has the form
$$
\xymatrix @R=-.2pc {
&0&\\
Q= \quad 1\; \bullet \ar@<.8ex>[r]^(.66){\beta} \ar@<1ex>[r];[]^(.34){\gamma}
& \bullet \ar@<.8ex>[r]^(.44){\delta} \ar@<1ex>[r];[]^(.56){\eta} & \bullet\; 2}
$$
and the ideals $I$ and $\overline{I}$  are given as
\begin{eqnarray*}
I&=&\langle \beta\gamma\beta-\eta\delta\beta
(\gamma\eta\delta\beta)^{2^{d-1}-1},\gamma\beta\gamma-\gamma\eta\delta (\beta\gamma\eta
\delta)^{2^{d-1}-1}, \\ 
&&\eta\delta\eta-\beta\gamma\eta(\delta\beta\gamma\eta)^{2^{d-1}-1},
\delta\eta\delta-\delta\beta\gamma(\eta\delta\beta\gamma)^{2^{d-1}-1},\\
&&\delta\beta\gamma\beta,\gamma\eta\delta\eta\rangle,\\
\overline{I}&=&\langle \gamma\beta,\delta\eta,(\eta\delta\beta\gamma)^{2^{d-2}}-(\beta\gamma\eta\delta
)^{2^{d-2}}\rangle.
\end{eqnarray*}
The radical series of the projective indecomposable $\Lambda$-modules 
$P_0$, $P_1$ and $P_2$ are described in  Figure \ref{fig:psl1B} 
where the radical series length of each of these modules is $2^{d+1}+1$. 
The decomposition matrix of $B$ is given in Figure \ref{fig:decompsl1}.

\subsection{Family (II)}
\label{ss:psl2A}
For Family (II), $d\ge 2$. The quiver $Q$ has the form
$$Q=\vcenter{\xymatrix  {
 0\,\bullet \ar@<.7ex>[rr]^{\beta} \ar@<.8ex>[rr];[]^{\gamma}\ar@<.7ex>[rdd]^{\kappa} \ar@<.8ex>[rdd];[]^{\lambda}
&&\bullet\ar@<.7ex>[ldd]^{\delta} \ar@<.8ex>[ldd];[]^{\eta}\,1\\&&\\ &
\unter{\mbox{\normalsize $\bullet$}}{\mbox{\normalsize $2$}}& }}$$
and the ideals $I$ and $\overline{I}$ are given as
\begin{eqnarray*}
I&=&\langle \delta\beta-\kappa\lambda\kappa,
\gamma\eta-\lambda\kappa\lambda,
\lambda\delta-\gamma\beta\gamma,
\eta\kappa-\beta\gamma\beta,\\
&&\beta\lambda-\eta(\delta\eta)^{2^{d-1}-1},\kappa\gamma-\delta(\eta\delta)^{2^{d-1}-1},\\
&&\delta\beta\gamma, \gamma\eta\delta, \eta\kappa\lambda
\rangle,\\
\overline{I}&=&\langle \delta\beta,\lambda\delta,\beta\lambda,\kappa\gamma,\eta\kappa,\gamma\eta,
\gamma\beta-\lambda\kappa,\kappa\lambda-(\delta\eta)^{2^{d-2}},(\eta\delta)^{2^{d-2}}-\beta\gamma
\rangle.
\end{eqnarray*}
The radical series of the projective indecomposable $\Lambda$-modules $P_0$, $P_1$ and $P_2$
are described in  Figure \ref{fig:psl2B}
where the radical series length of $P_0$ is $5$ and the radical series length of $P_1$ and $P_2$
is $2^{d}+1$. 
The decomposition matrix of $B$ is given in Figure \ref{fig:decompsl2}.

\subsection{Family (III)}
\label{ss:a7A}
For Family (III), $d\ge 3$. The quiver $Q$ has the form
$$\xymatrix @R=-.2pc {
&1&0&\\
 Q= \quad&\ar@(ul,dl)_{\alpha} \bullet \ar@<.8ex>[r]^{\beta} \ar@<.9ex>[r];[]^{\gamma}
& \bullet \ar@<.8ex>[r]^(.46){\delta} \ar@<.9ex>[r];[]^(.54){\eta} & \bullet\;2}$$
and the ideals $I$ and $\overline{I}$ are given as
\begin{eqnarray*}
I&=&\langle \beta\alpha-\eta\delta\beta(\gamma\eta\delta\beta),
\alpha\gamma-\gamma\eta\delta(\beta\gamma\eta\delta), \\
&& \delta\eta\delta-\delta\beta\gamma(\eta\delta\beta\gamma),
\eta\delta\eta-\beta\gamma\eta(\delta\beta\gamma\eta),\\
&& \gamma\beta-\alpha^{2^{d-1}-1},\beta\alpha^2,\delta\eta\delta\beta\rangle,\\
\overline{I}&=&\langle \beta\alpha,\alpha\gamma, \gamma\beta,\delta\eta,\eta\delta\beta\gamma-
\beta\gamma\eta\delta,\alpha^{2^{d-2}}-\gamma\eta\delta\beta\rangle.
\end{eqnarray*}
The radical series of the projective indecomposable $\Lambda$-modules $P_0$, $P_1$ and $P_2$ 
are described in  Figure \ref{fig:a7B} 
where the radical series length of $P_0$ and $P_2$ is $9$ and the radical series length of $P_1$
is $9$ if $d=3$ and $2^{d-1}+1$ if $d\ge 4$.
The decomposition matrix of $B$ is given in Figure \ref{fig:decoma7}.

\begin{rem}
\label{rem:oherdmann}
It follows from \cite{erd} that we have the following Morita equivalences for the blocks
$\overline{B}$ and $B$: 
\begin{enumerate}
\item[(i)] In family (I), $\overline{B}$ is Morita equivalent to the principal $2$-modular block of 
$\mathrm{PSL}_2(\mathbb{F}_q)$ and $B$ is Morita equivalent to the principal $2$-modular 
block of $\mathrm{SL}_2(\mathbb{F}_q)$ when $q$ is a prime power with $q\equiv 1 \mod 4$ such that
$2^{d+1}$ is the maximal $2$-power dividing $(q^2-1)$.
\item[(ii)] In family (II), $\overline{B}$ is Morita equivalent to the principal $2$-modular block of 
$\mathrm{PSL}_2(\mathbb{F}_q)$ and $B$ is Morita equivalent to the principal $2$-modular 
block of $\mathrm{SL}_2(\mathbb{F}_q)$ when $q$ is a prime power with $q\equiv 3 \mod 4$ such that
$2^{d+1}$ is the maximal $2$-power dividing $(q^2-1)$.
\item[(iii)] In family (III), if $d=3$ then $\overline{B}$ is Morita equivalent to the principal $2$-modular 
block of the alternating group $A_7$ and $B$ is Morita equivalent to the principal $2$-modular block 
of the non-trivial double cover $\tilde{A}_7$ of  $A_7$. If $d\ge 4$, by \cite[\S X.4]{erd}, it remains
open whether there are blocks $\overline{B}$ and $B$ that are Morita equivalent to the algebras
$\overline{\Lambda}$ and $\Lambda$, respectively.
\end{enumerate}
If $\overline{B}$ is the principal block of $k[G/N]$, then we can exclude the blocks in family (III) 
with $d\ge 4$ as follows. Since $N$ contains $O_{2'}(G)$ by Remark \ref{rem:principal}, $G/N$ 
has no non-trivial normal subgroup of odd order. 
By \cite{gowa}, it then follows that $G/N$ is isomorphic to either a subgroup of 
$\mathrm{P\Gamma L}_2(\mathbb{F}_q)$ containing $\mathrm{PSL}_2(\mathbb{F}_q)$ for some odd 
prime power $q$, or to the alternating group $A_7$. Using a theorem by Clifford 
\cite[Hauptsatz V.17.3]{hup}, we see that the only possibility for $\overline{B}$ to be Morita 
equivalent to a block in family (III) occurs when $d=3$ and $\overline{B}$ is Morita 
equivalent to the principal $2$-modular block of the alternating group $A_7$.
Therefore, if $d\ge 3$ then the blocks $\overline{B}$ in the families  (I), (II) and (III) that are Morita
equivalent to principal blocks are precisely the blocks considered in \cite{3sim}.
\end{rem}

In the following remark we introduce the notation that will enable us to go back and forth between
$B,\overline{B}$ and $\Lambda,\overline{\Lambda}$.

\begin{rem}
\label{rem:oyvey}
Assume Hypothesis $\ref{hyp:alltheway}$. By Remark \ref{rem:trythis}, the natural
projection $\pi:kG\to k[G/N]$ induces by restriction to $B$
a surjective $k$-algebra homomorphism $\pi_B: B\to \overline{B}$. Let $e$ be a
sum of orthogonal primitive idempotents in $B$ such that $eBe$ is basic and Morita equivalent to $B$, 
and let $\overline{e}=\pi_B(e)$. Then $\overline{e}\overline{B}\overline{e}$ is basic and Morita
equivalent to $\overline{B}$, and 
the restriction of $\pi_B$ to $eBe$ gives a surjective $k$-algebra homomorphism $\pi_e:eBe\to 
\overline{e}\overline{B}\overline{e}$. 

Suppose $\Lambda=kQ/I$ and $\overline{\Lambda}=kQ/\overline{I}$ are
as described in 
\S\ref{ss:psl1A}, \S\ref{ss:psl2A} or \S\ref{ss:a7A}
such that $B$ is Morita equivalent to
$\Lambda$ and $\overline{B}$ is Morita equivalent to $\overline{\Lambda}$.
Then there are $k$-algebra isomorphisms $f_{\Lambda}:eBe\to \Lambda$ and 
$f_{\overline{\Lambda}}: \overline{e}\overline{B}\overline{e}\to \overline{\Lambda}$. Moreover,
$\pi_e$ induces a surjective $k$-algebra homomorphism 
$\pi_\Lambda:\Lambda\to \overline{\Lambda}$ such that the 
following diagram commutes:
\begin{equation}
\label{eq:diagram}
\xymatrix {
B\ar[rr]^{\pi_B}&&\overline{B}\\
eBe\ar[u]^{\mathrm{incl.}}\ar[rr]^{\pi_e}&& \overline{e}\overline{B}\overline{e}\ar[u]_{\mathrm{incl.}}\\
\Lambda\ar[u]^{f_\Lambda}\ar[rr]^{\pi_\Lambda}&&\overline{\Lambda}\ar[u]_{f_{\overline{\Lambda}}}
}
\end{equation}
By Remark \ref{rem:trythis}, we can identify the simple $\overline{B}$-modules with the simple
$B$-modules via inflation using $\pi_B$. This implies that the simple 
$\overline{e}\overline{B}\overline{e}$-modules can also be identified with the simple $eBe$-modules
via inflation using $\pi_e$, and hence the simple $\overline{\Lambda}$-modules can be 
identified with the simple $\Lambda$-modules via inflation using $\pi_\Lambda$.
For $i\in\{0,1,2\}$, let $e_i\in \Lambda$ (resp. $\overline{e}_i\in\overline{\Lambda}$)
be the primitive idempotent corresponding to the vertex $i$ in $Q$.
Using Lemma \ref{lem:projectives}
and the structure of the projective indecomposable $\Lambda$-modules, 
we see that we can replace $f_{\overline{\Lambda}}$ by
$f_{\overline{\Lambda}}\circ\Xi$ 
for a suitable $k$-algebra automorphism $\Xi$ of $\overline{\Lambda}$ induced by a quiver 
automorphism of $Q$ so as
to be able to assume that
$\overline{\Lambda}\otimes_{\Lambda,\pi_\Lambda} \Lambda e_i \cong
\overline{\Lambda}\, \pi_\Lambda(e_i)\cong  \overline{\Lambda}\overline{e}_i$ for $i\in\{0,1,2\}$. 
We make this assumption from now on. Hence we obtain for $i\in\{0,1,2\}$ that
the simple $\Lambda$-module $S_i=\Lambda e_i/\mathrm{rad}(\Lambda) e_i$ is isomorphic to the
inflation via $\pi_\Lambda$ of the simple $\overline{\Lambda}$-module 
$\overline{\Lambda} \overline{e}_i/\mathrm{rad}(\overline{\Lambda})\overline{e}_i$.
\end{rem}


\section{Stable endomorphism rings}
\label{s:stablend}
\setcounter{equation}{0}
\setcounter{figure}{0}

Assume Hypothesis \ref{hyp:alltheway}.
Let $V$ be a finitely generated $k[G/N]$-module and
denote its inflation to $kG$ via $\pi$ also by $V$. By Higman's criterion 
(see \cite[Thm. 1]{higman}), the $kG$-module endomorphisms of $V$ that
factor through projective $kG$-modules are precisely those in the image of the trace map
$\mathrm{Tr}_1^G:\mathrm{End}_k(V)\to \mathrm{End}_{kG}(V)$, where 
$\mathrm{Tr}_1^G(\psi)(v)=\sum_{g\in G}g\,\psi(g^{-1}v)$ for all $\psi\in\mathrm{End}_k(V)$
and all $v\in V$. Because $N$ acts trivially on $V$, $\mathrm{Tr}_1^N$ is multiplication by 
$\#N = 2\cdot (\# O_B)$. Hence $\mathrm{Tr}_1^N$ is zero, which implies that 
$\mathrm{Tr}_1^G=\mathrm{Tr}_N^G\circ \mathrm{Tr}_1^N$ is also zero. 
We obtain the following result.

\begin{prop}
\label{prop:stablend}
Assume Hypothesis $\ref{hyp:alltheway}$.
Suppose $V$ is an indecomposable $k[G/N]$-module belonging to $\overline{B}$,
and denote its inflation to $kG$ via $\pi$ also by $V$. Then $V$ belongs to $B$, and
$\underline{\mathrm{End}}_{kG}(V)\cong \mathrm{End}_{k[G/N]}(V)$.
In particular, $\underline{\mathrm{End}}_{kG}(V)\cong k$ if and only if          
$\mathrm{End}_{k[G/N]}(V)\cong k$. 
\end{prop}

We now describe all $k[G/N]$-modules belonging to $\overline{B}$ whose endomorphism rings
are isomorphic to $k$ by describing the corresponding $\overline{\Lambda}$-modules and their
inflations via $\pi_\Lambda$ to $\Lambda$-modules, using the notation introduced in Remark \ref{rem:oyvey}.
We first need to define some indecomposable $\Lambda$-modules.

\begin{dfn}
\label{def:dirstrings}
Let $\Lambda=kQ/I$ where $Q$ and $I$ are either as in \S\ref{ss:psl1A},   \S\ref{ss:psl2A} or \S\ref{ss:a7A}. 
\begin{enumerate}
\item[(i)] Let $w=\zeta_n\cdots\zeta_2\zeta_1$ be a path of length $n\ge 1$ in $kQ$ 
whose image modulo $I$ does not lie in $\mathrm{soc}_2(\Lambda)$. 
For $1\le j\le n$, let $v_j$ be the end vertex of $\zeta_j$, and let $v_0$  be the starting
vertex of $\zeta_1$. Define a $kQ$-module $M_w$ of $k$-dimension $n+1$ with respect 
to a given $k$-basis $\{b_0,\ldots,b_n\}$ as follows. Let $0\le j\le n$.
If $v$ is a vertex in $Q$, define $v b_j=b_j$ if $v=v_j$, and $v b_j=0$ otherwise. If
$\zeta$ is an arrow in $Q$, define $\zeta b_j=b_{j+1}$ if $\zeta=\zeta_{j+1}$ and $j\le n-1$, 
otherwise define $\zeta b_j=0$. By our assumption on $w$, the ideal $I$ of $kQ$ acts as zero on $M_w$.
Hence $M_w$ defines a $\Lambda$-module, which we also denote by $M_w$.
Moreover, $M_w$ is a uniserial $\Lambda$-module whose descending
composition factors are $(S_{v_0},S_{v_1},\ldots,S_{v_n})$.
\item[(ii)] Let $\zeta_1,\zeta_2$ be two arrows in $Q$ which start (resp. end) at the same vertex $v_1$.
Let $v_0, v_2$ be the end vertices (resp. starting vertices) of $\zeta_1,\zeta_2$. 
Let $w=\zeta_2\zeta_1^{-1}$ (resp. $w=\zeta_2^{-1}\zeta_1$).
Define a $kQ$-module $M_w$ of $k$-dimension $3$ with respect to a given $k$-basis $\{b_0,b_1,b_2\}$
as follows. If $v$ is a vertex in $Q$ and $j\in\{0,1,2\}$, define $v b_j=b_j$ if $v=v_j$, and $v b_j=0$ otherwise. 
Define $\zeta_1 b_1 = b_0$ and $\zeta_2 b_1=b_2$ (resp. $\zeta_1 b_0=b_1$ and $\zeta_2 b_2=b_1$).
If $\zeta$ is an arrow in $Q$ and $j\in\{0,1,2\}$ such that
$(\zeta,j)\not\in\{(\zeta_1,1),(\zeta_2,1)\}$ (resp. $(\zeta,j)\not\in\{(\zeta_1,0),(\zeta_2,2)\}$),
define $\zeta b_j=0$. Since the ideal $I$ of $kQ$ acts as zero on $M_w$, this defines a 
$\Lambda$-module, which we also denote by $M_w$. Moreover, $M_w$ is a $\Lambda$-module with
$M_w/\mathrm{rad}(M_w)\cong S_{v_1}$ (resp. $\mathrm{soc}(M_w)\cong S_{v_1}$) and
$\mathrm{rad}(M_w)\cong S_{v_0}\oplus S_{v_2}$ (resp. $M/\mathrm{soc}(M_w)\cong S_{v_0}\oplus S_{v_2}$).
In particular, $\mathrm{soc}(M_w) = \mathrm{rad}(M_w)$.
\end{enumerate}
\end{dfn}

\begin{rem}
\label{rem:stablend}
Assume Hypothesis \ref{hyp:alltheway}. 
Suppose $\Lambda=kQ/I$ and $\overline{\Lambda}=kQ/\overline{I}$ are
as described in \S\ref{ss:psl1A}, \S\ref{ss:psl2A} or \S\ref{ss:a7A}
such that $B$ is Morita equivalent to $\Lambda$ and $\overline{B}$ is Morita 
equivalent to $\overline{\Lambda}$. Let $\pi_\Lambda:\Lambda\to\overline{\Lambda}$ 
be the surjective $k$-algebra
homomorphism from Remark \ref{rem:oyvey}. In particular, for $i\in\{0,1,2\}$,
the simple $\Lambda$-module $S_i$ is isomorphic to the
inflation via $\pi_\Lambda$ of the simple $\overline{\Lambda}$-module 
corresponding to the vertex $i$ in $Q$, and we denote the latter also by $S_i$.

We now describe the $\overline{\Lambda}$-modules whose endomorphism rings are isomorphic to $k$ 
and their inflations via $\pi_\Lambda$ to $\Lambda$-modules, up to isomorphism, using the following
notation:

A module of the form $\begin{array}{c}v_0\\[-.5ex] \vdots\\ v_n\end{array}$
denotes a uniserial $\overline{\Lambda}$-module whose factors in the descending 
radical series are  $(S_{v_0},\ldots,S_{v_n})$ and which is the unique such 
$\overline{\Lambda}$-module up to isomorphism.
A module of the form $\begin{array}{cc}\multicolumn{2}{c}{v_1}\\v_2&v_0\end{array}$ (resp. 
$\begin{array}{cc}v_2&v_0\\\multicolumn{2}{c}{v_1}\end{array}$) denotes an indecomposable
$\overline{\Lambda}$-module whose top (resp. socle) is simple and whose
factors in the descending radical series are
$(S_{v_1},S_{v_0}\oplus S_{v_2})$ (resp. $(S_{v_0}\oplus S_{v_2},S_{v_1})$)
and which is the unique such $\overline{\Lambda}$-module up to isomorphism.

\begin{enumerate}
\item[(i)]
If $\Lambda$ and $\overline{\Lambda}$ are as in \S$\ref{ss:psl1A}$, then a complete list of the 
$\overline{\Lambda}$-modules whose endomorphism ring is isomorphic to $k$, up to isomorphism, 
is given by
$$S_0,S_1, S_2,  \begin{array}{c}1\\0\end{array}, \begin{array}{c}0\\1\end{array},
\begin{array}{c}0\\2\end{array},\begin{array}{c}2\\0\end{array},
\begin{array}{c}1\\0\\2\end{array},\begin{array}{c}2\\0\\1\end{array},
\begin{array}{c}0\\1\\0\\2\end{array}, \begin{array}{c}0\\2\\0\\1\end{array},
\begin{array}{c}1\\0\\2\\0\end{array},\begin{array}{c}2\\0\\1\\0\end{array},
\begin{array}{cc}\multicolumn{2}{c}{0}\\1&2\end{array}, 
\begin{array}{cc}1&2\\\multicolumn{2}{c}{0}\end{array}.$$
Since $\mathrm{Ext}^1_\Lambda(S_0,S_i)\cong k \cong\mathrm{Ext}^1_\Lambda(S_i,S_0)$ for
$i\in\{1,2\}$, the inflations via $\pi_\Lambda$ of the two-dimensional $\overline{\Lambda}$-modules in this list 
are isomorphic to $M_\beta, M_\gamma, M_\delta, M_\eta$ (as defined in Definition \ref{def:dirstrings}).
Because $\mathrm{Ext}^1_\Lambda(M_\beta,S_2)\cong k\cong \mathrm{Ext}^1_\Lambda(M_\eta,S_1)$
and $\mathrm{Ext}^1_\Lambda(M_\gamma,S_2)\cong k\cong \mathrm{Ext}^1_\Lambda(S_1,M_\eta)$,
the inflations via $\pi_\Lambda$ of the three-dimensional $\overline{\Lambda}$-modules in this list 
are isomorphic to $M_{\delta\beta},M_{\gamma\eta}, M_{\gamma\delta^{-1}},M_{\beta^{-1}\eta}$. 
Since $\mathrm{Ext}^1_\Lambda(S_0,M_{\delta\beta})\cong 
k\cong\mathrm{Ext}^1_\Lambda(S_0,M_{\gamma\eta})$ and 
$\mathrm{Ext}^1_\Lambda(M_{\delta\beta},S_0)\cong k\cong 
\mathrm{Ext}^1_\Lambda(M_{\gamma\eta},S_0)$, the inflations via $\pi_\Lambda$ of the four-dimensional 
$\overline{\Lambda}$-modules in this list  are isomorphic to
$M_{\delta\beta\gamma}, M_{\gamma\eta\delta},M_{\eta\delta\beta},M_{\beta\gamma\eta}$.

Therefore, a complete list of the $\Lambda$-modules which are inflated via $\pi_\Lambda$
from the above $\overline{\Lambda}$-modules, up to isomorphism, is given by
$$S_0,S_1, S_2, M_\beta, M_\gamma, M_\delta, M_\eta, M_{\delta\beta},
M_{\gamma\eta}, M_{\delta\beta\gamma}, M_{\gamma\eta\delta},M_{\eta\delta\beta},
M_{\beta\gamma\eta}, M_{\gamma\delta^{-1}},M_{\beta^{-1}\eta}.$$

\item[(ii)] If $\Lambda$ and $\overline{\Lambda}$ are as in \S$\ref{ss:psl2A}$, then a complete list of the 
$\overline{\Lambda}$-modules whose endomorphism ring is isomorphic to $k$, up to isomorphism, is given by
$$S_0,S_1, S_2,  \begin{array}{c}0\\1\end{array}, \begin{array}{c}1\\0\end{array}, 
\begin{array}{c}1\\2\end{array},\begin{array}{c}2\\1\end{array},
\begin{array}{c}0\\2\end{array},\begin{array}{c}2\\0\end{array},
\begin{array}{cc}\multicolumn{2}{c}{0}\\1&2\end{array}, 
\begin{array}{cc}\multicolumn{2}{c}{1}\\0&2\end{array}, 
\begin{array}{cc}\multicolumn{2}{c}{2}\\0&1\end{array}, 
\begin{array}{cc}1&2\\\multicolumn{2}{c}{0}\end{array},
\begin{array}{cc}0&2\\\multicolumn{2}{c}{1}\end{array},
\begin{array}{cc}0&1\\\multicolumn{2}{c}{2}\end{array}.$$
Using similar $\mathrm{Ext}$ calculations as in (i), we see that
a complete list of the $\Lambda$-modules which are inflated via $\pi_\Lambda$
from these $\overline{\Lambda}$-modules, up to isomorphism, is given by
$$S_0,S_1, S_2, M_\beta, M_\gamma, M_\delta, M_\eta, M_\kappa, M_\lambda,
M_{\beta\kappa^{-1}}, M_{\gamma\delta^{-1}}, M_{\lambda\eta^{-1}},
M_{\gamma^{-1}\lambda}, M_{\beta^{-1}\eta}, M_{\kappa^{-1}\delta}.$$

\item[(iii)] 
If $\Lambda$ and $\overline{\Lambda}$ are as in \S$\ref{ss:a7A}$, a complete list of the 
$\overline{\Lambda}$-modules whose endomorphism ring is isomorphic to $k$, up to isomorphism, is 
given by
$$S_0,S_1, S_2,  \begin{array}{c}1\\0\end{array}, \begin{array}{c}0\\1\end{array},
\begin{array}{c}0\\2\end{array},\begin{array}{c}2\\0\end{array},
\begin{array}{c}1\\0\\2\end{array},\begin{array}{c}2\\0\\1\end{array},
\begin{array}{c}0\\1\\0\\2\end{array}, \begin{array}{c}0\\2\\0\\1\end{array},
\begin{array}{c}1\\0\\2\\0\end{array},\begin{array}{c}2\\0\\1\\0\end{array},
\begin{array}{cc}\multicolumn{2}{c}{0}\\1&2\end{array}, 
\begin{array}{cc}1&2\\\multicolumn{2}{c}{0}\end{array}.$$
Using similar $\mathrm{Ext}$ calculations as in (i), we see that
a complete list of the $\Lambda$-modules which are inflated via $\pi_\Lambda$
from these $\overline{\Lambda}$-modules, up to isomorphism, is given by
$$S_0,S_1, S_2, M_\beta, M_\gamma, M_\delta, M_\eta, M_{\delta\beta},
M_{\gamma\eta}, M_{\delta\beta\gamma}, M_{\gamma\eta\delta},M_{\eta\delta\beta},
M_{\beta\gamma\eta}, M_{\gamma\delta^{-1}},M_{\beta^{-1}\eta}.$$
\end{enumerate}
\end{rem}


\section{Ordinary characters for blocks with generalized quaternion defect groups}
\label{s:ordinaryquaternion}
\setcounter{equation}{0}
\setcounter{figure}{0}

Assume Hypothesis \ref{hyp:alltheway}.
In the notation of \cite[\S2]{olsson}, this means that we 
are in Case $(aa)$ (see \cite[Thm. 3.17]{olsson}).
Let $W$ be the ring of infinite Witt vectors over $k$, and
let $F$ be the fraction field of $W$.
For $2\le\ell\le d$, let $\zeta_{2^\ell}$ be a fixed primitive $2^\ell$-th root of unity 
in an algebraic closure of $F$. Let
$$\chi_1,\chi_2,\chi_3,\chi_4,\chi_5,\chi_6,\qquad \chi_{7,i}, 1\le i\le 2^{d-1}-1,$$
be the ordinary irreducible characters of $G$ belonging to $B$. Let $\sigma$ be an element of order 
$2^d$ in $D$. By \cite{olsson}, there is a block $b_\sigma$ of $kC_G(\sigma)$ with 
$b_\sigma^G=B$ which contains a unique $2$-modular character $\varphi^{(\sigma)}$ such that 
the following is true. There is an ordering of $(1,2,\ldots,2^{d-1}-1)$ such that for 
$1\le i\le 2^{d-1}-1$ and $r$ odd,
\begin{equation}
\label{eq:great1}
\chi_{7,i}(\sigma^r)=(\zeta_{2^{d}}^{ri}+\zeta_{2^{d}}^{-ri})\cdot \varphi^{(\sigma)}(1).
\end{equation}

Note that $W$ contains all roots of unity of order not divisible by $2$. Hence by \cite{olsson} and by 
\cite{fong}, the characters 
$\chi_1,\chi_2,\chi_3,\chi_4,\chi_5,\chi_6$
correspond to simple $FG$-modules. On the other hand, the characters 
$\chi_{7,i}$, $i=1,\ldots,2^{d-1}-1$,
fall into $d-1$ Galois orbits $\mathcal{O}_2,\ldots, \mathcal{O}_d$ under the action of 
$\mathrm{Gal}(F(\zeta_{2^d}+\zeta_{2^d}^{-1})/F)$. Namely for $2\le\ell\le d$, 
$\mathcal{O}_{\ell}=\{ \chi_{7,2^{d-\ell}(2u-1)} \;|\; 1\le u\le 2^{\ell-2}\}$. The field generated by the 
character values of each $\xi_\ell\in\mathcal{O}_\ell$ over $F$ is $F(\zeta_{2^\ell}+\zeta_{2^\ell}^{-1})$. 
Hence by \cite{fong}, each $\xi_\ell$ corresponds to an absolutely irreducible 
$F(\zeta_{2^\ell}+\zeta_{2^\ell}^{-1})G$-module $X_\ell$. 
By \cite[Satz V.14.9]{hup},  this implies that for $2\le\ell\le d$, the Schur index of each 
$\xi_\ell\in\mathcal{O}_\ell$ over $F$ is $1$. Hence we obtain $d-1$ non-isomorphic simple 
$FG$-modules $V_2,\ldots,V_{d}$
with characters $\rho_2,\ldots, \rho_{d}$ satisfying
\begin{equation}
\label{eq:goodchar1}
\rho_\ell =\sum_{\xi_\ell\in\mathcal{O}_\ell}\xi_\ell = 
\sum_{u=1}^{2^{\ell-2}} \chi_{7,2^{d-\ell}(2u-1)} 
\qquad\mbox{for $2\le \ell \le d$.}
\end{equation}
By \cite[Hilfssatz V.14.7]{hup}, $\mathrm{End}_{FG}(V_\ell)$ is a commutative $F$-algebra isomorphic 
to the field generated over $F$ by the character values of any $\xi_\ell\in\mathcal{O}_\ell$. 
This means
\begin{equation}
\label{eq:goodendos}
\mathrm{End}_{FG}(V_\ell)\cong F(\zeta_{2^\ell}+\zeta_{2^\ell}^{-1})\qquad\mbox{for $2\le \ell \le d$.}
\end{equation}

By \cite{olsson}, the characters $\chi_{7,i}$ have the same degree $x$ for $1\le i\le 2^{d-1}-1$. The characters $\chi_1,\chi_2,\chi_3,\chi_4$ have height $0$, $\chi_5, \chi_6$ have height $d-1$,
and $\chi_{7,i}$, $1\le i\le 2^{d-1}-1$, have height $1$. 
Hence $x=2^{a-(d+1)+1}x^*$ where $\#G=2^a\cdot g^*$ and $x^*$ and $g^*$ are odd. Since the 
centralizer $C_G(\sigma)$ contains $\langle \sigma \rangle$, we have $\# C_G(\sigma)=
2^d\cdot 2^b\cdot m^*$ where $b\ge 0$ and $m^*$ is odd. Suppose $\varphi^{(\sigma)}(1)=
2^c\cdot n^*$ where $c\ge 0$ and $n^*$ is odd. Note that if $\psi$ is an ordinary irreducible 
character of $C_G(\sigma)$ belonging to the block $b_\sigma$, then by \cite[p. 61]{serre}, $\psi(1)$ 
divides $(\# C_G(\sigma))/(\#\langle\sigma\rangle)=2^b\cdot m^*$. Because $\psi(1)=
s_\psi\cdot \varphi^{(\sigma)}(1)$ for some positive integer $s_\psi$, we have $c\le b$. 

Let $C$ be the conjugacy class in $G$ of $\sigma$, and let $t(C)\in WG$ be the class sum of $C$. We 
want to determine the action of $t(C)$ on $V_\ell$ for $2\le \ell\le d$. For this, we identify 
$\mathrm{End}_{FG}(V_\ell)\cong F(\zeta_{2^\ell}+\zeta_{2^\ell}^{-1})$ with 
$\mathrm{End}_{F(\zeta_{2^\ell}+\zeta_{2^\ell}^{-1})G}(X_\ell)$ for one particular absolutely irreducible 
$F(\zeta_{2^\ell}+\zeta_{2^\ell}^{-1})G$-constituent $X_\ell$ of $V_\ell$ with character $\xi_\ell$. 
By (\ref{eq:goodchar1}), we can choose $\xi_\ell=\chi_{7,2^{d-\ell}}$. Then, under this identification, 
for $2\le \ell \le d$, the action of $t(C)$ on $V_\ell$ is given as multiplication by 
\begin{equation}
\label{eq:longone}
\frac{\# C}{\xi_\ell(1)} \cdot \xi_\ell(\sigma)
=2^{c-b} \frac{g^*\cdot n^*}{m^*\cdot x^*}\cdot (\zeta_{2^d}^{2^{d-\ell}}+\zeta_{2^d}^{-2^{d-\ell}})
\end{equation}
where, as shown above, $c\le b$. Note that $\frac{g^*\cdot n^*}{m^*\cdot x^*}$ is a unit in $W$, since 
$g^*\cdot n^*$ and $m^*\cdot x^*$ are odd. Since $t(C)\in WG$, we must have $c\ge b$, i.e. $c=b$.
Therefore, (\ref{eq:longone}) implies that there exists a unit $\omega$ in $W$ such that for 
$2\le \ell \le d$, the action of $t(C)$ on $V_\ell$ is given as multiplication by
\begin{equation}
\label{eq:thatsit}
\omega\cdot (\zeta_{2^d}^{2^{d-\ell}}+\zeta_{2^d}^{-2^{d-\ell}})
\end{equation}
when we identify $\mathrm{End}_{FG}(V_\ell)$ with $\mathrm{End}_{F(\zeta_{2^\ell}+\zeta_{2^\ell}^{-1})
G}(X_\ell)$ for an absolutely irreducible $F(\zeta_{2^\ell}+\zeta_{2^\ell}^{-1})G$-constituent $X_\ell$ of 
$V_\ell$ whose character is $\chi_{7,2^{d-\ell}}$.

\begin{dfn}
\label{def:seemtoneed}
Use the notation introduced above.
\begin{enumerate}
\item[(i)] Define
$$p_{d+1}(t)=\prod_{\ell=2}^d \mathrm{min.pol.}_F(\zeta_{2^\ell}+\zeta_{2^\ell}^{-1}),$$
and let $R'=W[[t]]/(p_{d+1}(t))$.
\item[(ii)] Let $Z=\langle \sigma\rangle$ be a cyclic group of order $2^d$, and let $\tau:Z\to Z$ be the group automorphism sending $\sigma$ to $\sigma^{-1}$. Then $\tau$ can be extended to a $W$-algebra automorphism of the group ring $WZ$ which will again be denoted by $\tau$. Let $T(\sigma^2)=1+\sigma^2+\sigma^4+\cdots +\sigma^{2^d-2}$, and define
$$S'= (WZ)^{\langle \tau\rangle}/\left(T(\sigma^2),\sigma T(\sigma^2)\right).$$
\end{enumerate}
\end{dfn}

\begin{rem}
\label{rem:ohyeah}
The minimal polynomial $\mathrm{min.pol.}_F(\zeta_{2^\ell}+\zeta_{2^\ell}^{-1})$ for $\ell\ge 2$ is as follows:
\begin{eqnarray*}
\mathrm{min.pol.}_F(\zeta_{2^2}+\zeta_{2^2}^{-1})(t)&=&t,\\ \mathrm{min.pol.}_F(\zeta_{2^{\ell}}+\zeta_{2^{\ell}}^{-1})(t)
&=&\left(\mathrm{min.pol.}_F(\zeta_{2^{\ell-1}}+\zeta_{2^{\ell-1}}^{-1})(t)\right)^2-2 \qquad \mbox{for $\ell\ge 3$}.
\end{eqnarray*}

The $W$-algebra $R'$ from Definition \ref{def:seemtoneed} is a complete local 
commutative Noetherian ring with residue field $k$. Moreover, 
\begin{eqnarray*}
F\otimes_W R'&\cong& \prod_{\ell=2}^d F(\zeta_{2^\ell}+\zeta_{2^\ell}^{-1})\quad\mbox{ as $F$-algebras,}\\
k\otimes_W R'&\cong& k[t]/(t^{2^{d-1}-1})\quad\mbox{ as $k$-algebras.}
\end{eqnarray*}
Additionally, for any sequence $\left(r_\ell\right)_{\ell=2}^d$ of odd integers, $R'$ is isomorphic to the $W$-subalgebra of
$$\prod_{\ell=2}^d W[\zeta_{2^\ell}+\zeta_{2^\ell}^{-1}]$$
generated by the element 
$\left(\zeta_{2^\ell}^{r_\ell}+\zeta_{2^\ell}^{-r_\ell}\right)_{\ell=2}^d$.
\end{rem}

\begin{lemma}
\label{lem:quaterniontrick}
Using the notation of Definition $\ref{def:seemtoneed}$, there 
is a continuous $W$-algebra isomorphism $\rho:R'\to S'$ with $\rho(t)=\sigma+\sigma^{-1}$. 
In particular, $R'$ is isomorphic to a subquotient algebra of the group ring over $W$ of a
generalized quaternion group of order $2^{d+1}$.
\end{lemma}

\begin{proof}
It follows from \cite[Lemma 2.3.6]{3sim} that $\rho:R'\to S'$ is a continuous $W$-algebra isomorphism.
Hence the description of $S'$ in Definition \ref{def:seemtoneed} shows that $S'$ is isomorphic to a
subquotient algebra of the group ring over $W$ of a generalized quaternion group of order $2^{d+1}$.
\end{proof}

\begin{lemma}
\label{lem:quaternionargument}
Assume Hypothesis $\ref{hyp:alltheway}$ and use the notation introduced above.
Let $U'$ be a $WG$-module which is free over $W$ and whose $F$-character is
$$\sum_{\ell=2}^d \rho_\ell = \sum_{i=1}^{2^{d-1}-1}\chi_{7,i}.$$
Then $U'$ is an $R'G$-module which is free as an $R'$-module and $\mathrm{End}_{WG}(U')
\cong R'$.
\end{lemma}

\begin{proof}
We first prove that $R'$ is isomorphic to a $W$-subalgebra of 
$\mathrm{End}_{WG}(U')$. Let $\sigma$ be the same element of order $2^d$ in $D$ as in 
(\ref{eq:great1}), and let $t(C)$ be the 
class sum of the conjugacy class $C$ of $\sigma$ in $G$. Since $t(C)$ lies in the center of $WG$, 
multiplication by $t(C)$ defines a $WG$-module endomorphism of $U'$. Since 
$\mathrm{End}_{WG}(U')$ can naturally be identified with a subring of 
$\mathrm{End}_{FG}(F\otimes_W U')$, $t(C)$ acts on $U'$ as multiplication by a scalar $\lambda_C$ in 
$F\otimes_W R'\cong\prod_{\ell=2}^dF(\zeta_{2^\ell}+\zeta_{2^\ell}^{-1})$. Moreover, $\lambda_C$ 
can be read off from the action of $t(C)$ on $F\otimes_W U'\cong\bigoplus_{\ell=2}^dV_\ell$.
By $(\ref{eq:thatsit})$, $\lambda_C$ equals
$\omega\cdot (\zeta_{2^\ell}^{r_\ell}+\zeta_{2^\ell}^{-r_{\ell}})_{\ell=2}^d$
for some unit $\omega\in W$ and some sequence $(r_\ell)_{\ell=2}^d$ of odd integers.
Since $R'$ is isomorphic to the 
$W$-subalgebra of $\prod_{\ell=2}^d W[\zeta_{2^\ell}+\zeta_{2^\ell}^{-1}]$ generated by 
$(\zeta_{2^\ell}^{r_\ell}+\zeta_{2^\ell}^{-r_{\ell}})_{\ell=2}^{d}$,
this implies that $R'$ is isomorphic to a $W$-subalgebra of 
$\mathrm{End}_{WG}(U')$.
Hence $U'$ is an $R'G$-module. 

We next prove that $U'$ is free as an 
$R'$-module. Since $U'$ is finitely generated as a $W$-module, it is also finitely generated as an 
$R'$-module. Since $R'$ is a local ring with maximal ideal $m_{R'}$, it follows by Nakayama's Lemma that any $k$-basis $\{\overline{b}_1,\ldots, \overline{b}_s\}$ of $U'/m_{R'}U'$ can be lifted to a set 
$\{b_1,\ldots, b_s\}$ of generators of $U'$ over $R'$. Since $F\otimes_W U'$ is a free 
$(F\otimes_W R')$-module of rank $s$, it follows that $b_1,\ldots, b_s$ are linearly independent over 
$R'$. Because $\mathrm{End}_{FG}(F\otimes_W U')\cong F\otimes_WR'$, this then implies that 
$\mathrm{End}_{WG}(U')\cong R'$. 
\end{proof}


\section{Universal deformation rings}
\label{s:udr}
\setcounter{equation}{0}
\setcounter{figure}{0}

As in \S\ref{s:ordinaryquaternion}, let $W$ be the ring of infinite Witt vectors over $k$ 
and let $F$ be the fraction field of $W$.

\begin{thm}
\label{thm:altogether}
Assume Hypothesis $\ref{hyp:alltheway}$.
Suppose $B$ is Morita equivalent to $\Lambda=kQ/I$ where $Q$ and $I$
are as in 
\S$\ref{ss:psl1A}$ $($resp. \S$\ref{ss:psl2A}$, resp. \S$\ref{ss:a7A}$$)$.
Denote the three simple $B$-modules 
by $T_0$, $T_1$ and $T_2$, where $T_i$ corresponds to $S_i$, for 
$i\in\{0,1,2\}$, under the Morita equivalence between $B$ and $\Lambda$.
Suppose $M$ is an indecomposable $B$-module whose stable endomorphism ring
is isomorphic to $k$ and which is inflated from a $\overline{B}$-module.
For $d\ge 2$, let $p_{d+1}(t)\in W[t]$ be as in Definition $\ref{def:seemtoneed}$.
\begin{enumerate}
\item[(i)] Suppose $\Lambda$ is as in 
\S$\ref{ss:psl1A}$
and $d\ge 2$.
If $M$ is a uniserial module of length $4$, then 
$R(G,M)\cong W[[t]]/(p_{d+1}(t))$.
Otherwise $R(G,M)\cong W$.
\item[(ii)] Suppose $\Lambda$ is as in 
\S$\ref{ss:psl2A}$
and $d\ge 2$.
If $M$ is a uniserial module of length $2$ with composition factors $T_1,T_2$, 
then $R(G,M)\cong W[[t]]/(p_{d+1}(t))$.
Otherwise $R(G,M)\cong W$.
\item[(iii)] Suppose $\Lambda$ is as in 
\S$\ref{ss:a7A}$ and $d\ge 3$.
If $M$ is isomorphic to $T_1$, then $R(G,M)\cong W[[t]]/(p_{d+1}(t))$.
Otherwise $R(G,M)\cong W$.
\end{enumerate}
In all cases $(i)-(iii)$, $R(G,M)$ is isomorphic to a subquotient ring of $WD$ and a complete intersection.
\end{thm}

\begin{proof}
By Proposition \ref{prop:stablend}, $M$ is inflated from a $\overline{B}$-module whose endomorphism 
ring is isomorphic to $k$. Hence we can use
the lists given in Remark \ref{rem:stablend} to describe the possibilities for $M$. Recall that
we have shown in Remark \ref{rem:stablend} that all these $M$ are uniquely determined
by the factors in their descending radical series.

We prove part (i) of Theorem \ref{thm:altogether}, the proofs of parts (ii) and (iii) being similar.
Using the list in Remark \ref{rem:stablend}(i) and the
description of the projective indecomposable $B$-modules in Figure \ref{fig:psl1B},
it follows for all $d\ge 2$
that $\mathrm{Ext}^1_{kG}(M,M)=0$ if $M$ is not a uniserial module of length $4$.
If $M$ is uniserial of length $4$, then $\mathrm{Ext}^1_{kG}(M,M)=0$ if $d=2$ and
$\mathrm{Ext}^1_{kG}(M,M)\cong k$ if $d\ge 3$.

We first show that each $M$ corresponding to a $\Lambda$-module
in the list in Remark \ref{rem:stablend}(i) has a lift over $W$.
We see directly from the decomposition matrix in Figure \ref{fig:decompsl1} that
if $M$ is simple then $M$ has a lift over $W$. We divide the remaining $M$ into two subsets
$\mathcal{M}_1$ and $\mathcal{M}_2$ where
$$\mathcal{M}_1=\left\{\begin{array}{c}T_1\\T_0\end{array}, \begin{array}{c}T_0\\T_1\end{array},
\begin{array}{c}T_0\\T_2\end{array},\begin{array}{c}T_2\\T_0\end{array},
\begin{array}{c}T_1\\T_0\\T_2\end{array},\begin{array}{c}T_2\\T_0\\T_1\end{array},
\begin{array}{c}T_1\\T_0\\T_2\\T_0\end{array},\begin{array}{c}T_2\\T_0\\T_1\\T_0\end{array},
\begin{array}{cc}\multicolumn{2}{c}{T_0}\\T_1&T_2\end{array}\right\}\quad\mbox{ and }$$
$$\mathcal{M}_2=\left\{\begin{array}{c}T_0\\T_1\\T_0\\T_2\end{array}, \begin{array}{c}T_0\\T_2\\T_0\\T_1\end{array},
\begin{array}{cc}T_1&T_2\\\multicolumn{2}{c}{T_0}\end{array}\right\}.$$

For $M\in\mathcal{M}_1$, let $P^W_M$ be a projective indecomposable $WG$-module
such that $k\otimes_WP^W_M$ is a projective $kG$-module cover of $M$.
Using the decomposition matrix in Figure \ref{fig:decompsl1} and \cite[Prop. (23.7)]{CR}, it follows
for $M\in\mathcal{M}_1$
that there is a $W$-pure $WG$-sublattice $X_M$ of  $P^W_M$
such that $U_M=P^W_M/X_M$ defines a lift over $W$ of a $B$-module $N$ which has 
the same top and the same composition factors as $M$. Moreover, it follows from the
description of the projective indecomposable $B$-modules in Figure \ref{fig:psl1B} that 
the factors in the descending radical series of $N$ and $M$ are the same.
Hence it follows from Remark \ref{rem:stablend}(i) that $N$ has to be isomorphic to $M$.

For $M\in\mathcal{M}_2$, let $Q^W_M$ be the projective indecomposable $WG$-module
such that $k\otimes_WQ^W_M$ is a projective $kG$-module cover of $\Omega^{-1}(M)$. 
Using the decomposition matrix in Figure \ref{fig:decompsl1}, \cite[Prop. (23.7)]{CR}
and the description of the projective indecomposable $B$-modules in Figure \ref{fig:psl1B}, 
it follows for $M\in\mathcal{M}_2$
that there is a $W$-pure $WG$-sublattice $Y_M$ of  $Q^W_M$
such that $V_M=Q^W_M/Y_M$ defines a lift of $\Omega^{-1}(M)$ over $W$.
Therefore, $Y_M$ defines a lift of $M$ over $W$.

Hence, every 
$B$-module $M$ which is inflated from a $\overline{B}$-module whose endomorphism ring 
is isomorphic to $k$ has a lift over $W$. This implies that  $R(G,M)\cong W$ 
for all such $M$ in case $d=2$, and for all such $M$ that are not uniserial of length $4$
in case $d\ge 3$.
In particular, this proves part (i) of Theorem \ref{thm:altogether} in case $d=2$ since
$p_3(t)=t$.

To finish the proof of part (i) in general,
assume now
that $d\ge 3$ and $M$ is a uniserial module of length $4$. 
Let $i\neq j$ in $\{1,2\}$ and define
$$M_{i0j0}=\begin{array}{c}T_i\\T_0\\T_j\\T_0\end{array}\qquad\mbox{and}\qquad 
M_{0i0j}=\begin{array}{c}T_0\\T_i\\T_0\\T_j\end{array}.$$
Then $M\cong M_{i0j0}$ (resp.  $M\cong M_{0i0j}$) for some choice of
$i\neq j$ in $\{1,2\}$, and $\mathrm{Ext}^1_{kG}(M,M)\cong k$. 
Considering the projective $B$-module cover $P_{T_i}$ of $T_i$, 
it follows that $P_{T_i}$ has a uniserial quotient module 
$\overline{U}_{i0j0}$ of length $4\cdot (2^{d-1}-1)$ whose
composition factors in the descending radical series are given as
$$(T_i,T_0,T_j,T_0,T_i,T_0,T_j,T_0,\ldots,T_i,T_0,T_j,T_0)$$
such that there are $kG$-module isomorphisms
\begin{eqnarray*}
&\phi_{i0j0}: & \overline{U}_{i0j0} /\mathrm{rad}^{4} (\overline{U}_{i0j0})
\to M_{i0j0}\,,\\
&\psi_{i0j0}: & \overline{U}_{i0j0}/\mathrm{rad}^{4(2^{d-1}-2)} 
(\overline{U}_{i0j0})\to \mathrm{rad}^{4} (\overline{U}_{i0j0}), 
\end{eqnarray*}
and $\mathrm{Ext}^1_{kG}(\overline{U}_{i0j0},M_{i0j0})=0$.
Namely, the $B$-module $\overline{U}_{1020}$ (resp. $\overline{U}_{2010}$)
corresponds to the $\Lambda$-module $M_{\eta\delta\beta(\gamma\eta\delta\beta)^{2^{d-1}-2}}$ 
(resp. $M_{\beta\gamma\eta(\delta\beta\gamma\eta)^{2^{d-1}-2}}$), as defined in Definition 
\ref{def:dirstrings}, under the Morita equivalence between $B$ and $\Lambda$.
Similarly, the projective $B$-module cover $P_{T_0}$ of $T_0$
has a uniserial quotient module $\overline{U}_{0i0j}$ of length $4\cdot (2^{d-1}-1)$ whose 
composition factors in the descending radical series are given as
$$(T_0,T_i,T_0,T_j,T_0,T_i,T_0,T_j,\ldots,T_0,T_i,T_0,T_j)$$
such that there are $kG$-module isomorphisms
\begin{eqnarray*}
&\phi_{0i0j}: &\overline{U}_{0i0j} /\mathrm{rad}^{4} (\overline{U}_{0i0j})
\to M_{0i0j}\,,\\
&\psi_{0i0j}: &\overline{U}_{0i0j}/\mathrm{rad}^{4(2^{d-1}-2)} 
(\overline{U}_{0i0j}) \to \mathrm{rad}^{4} (\overline{U}_{0i0j}),
\end{eqnarray*}
and $\mathrm{Ext}^1_{kG}(\overline{U}_{0i0j},M_{0i0j})=0$.
Namely, the $B$-module $\overline{U}_{0102}$ (resp. $\overline{U}_{0201}$)
corresponds to the $\Lambda$-module $M_{\delta\beta\gamma(\eta\delta\beta\gamma)^{2^{d-1}-2}}$
(resp. $M_{\gamma\eta\delta(\beta\gamma\eta\delta)^{2^{d-1}-2}}$)
under the Morita equivalence between $B$ and $\Lambda$.
Hence it follows by Lemma \ref{lem:trythis} in case $M$ is isomorphic to $M_{i0j0}$ (resp. to
$M_{0i0j}$) 
that $\overline{R}=R(G,M)/2R(G,M)\cong k[t]/(t^{2^{d-1}-1})$. Moreover, the universal 
mod $2$ deformation of $M$ over $\overline{R}$ is  
$[\overline{U}_{i0j0},\phi_{i0j0}]$ (resp. $[\overline{U}_{0i0j},\phi_{0i0j}]$)
where the action of $t$ on $\overline{U}_{i0j0}$ (resp. $\overline{U}_{0i0j}$) is 
given by the $kG$-module endomorphism $\mu_{t,i0j0}$ (resp. $\mu_{t,0i0j}$)
of $\overline{U}_{i0j0}$ (resp. $\overline{U}_{0i0j}$)
which is induced by $\psi_{i0j0}$ (resp. $\psi_{0i0j}$).

We now use Lemma \ref{lem:Wlift} to show that $\overline{U}_{i0j0}$ (resp. $\overline{U}_{0i0j}$)
has a lift over $W$. Consider the following submodule $Z_{i0j0}$ (resp. quotient module
$Z_{0j0i}$) of the projective indecomposable $B$-module $P_{T_i}$:
$$Z_{i0j0}=\Omega(\overline{U}_{i0j0})=
\begin{array}{c@{}c@{}c@{}c@{}c} &&&&T_i\\&&&T_0\\T_i&&T_j\\&T_0\\&T_i
\end{array},\qquad \qquad Z_{0j0i}=\Omega^{-1}(\overline{U}_{0j0i})=
\begin{array}{c@{}c@{}c@{}c@{}c}&T_i\\&T_0\\T_i&&T_j\\&&&T_0\\
&&&&T_i\end{array}.$$
If $Y_{s0t}=\begin{array}{c}T_s\\T_0\\T_t\end{array}$ for $s,t\in\{1,2\}$ then
we have two non-split short exact sequences of $B$-modules
$$0\to Y_{i0i}\to Z_{i0j0}\to Y_{i0j}\to 0$$
and 
$$0\to Y_{j0i}\to Z_{0j0i}\to Y_{i0i}\to 0$$
where $\mathrm{Ext}^1_{kG}(Y_{i0j},Y_{i0i})\cong k$ and $\mathrm{Ext}^1_{kG}(Y_{i0i},Y_{j0i})\cong k$.
Moreover, it follows from the decomposition matrix in Figure \ref{fig:decompsl1} and from the
description of the projective indecomposable $B$-modules in Figure \ref{fig:psl1B} that $Y_{s0t}$
has a lift $(X_{s0t},\xi_{s0t})$ over $W$ for all $s,t\in\{1,2\}$ such that the following holds. 
If $s\neq t$ then the $F$-character of $X_{s0t}$ is $\chi_4$, the $F$-character
of $X_{101}$ is $\chi_2+\chi_5$, and the $F$-character of $X_{202}$ is 
$\chi_3+\chi_6$. Therefore, we have
$$\mathrm{Hom}_{FG}(F\otimes_W X_{i0j},F\otimes_W X_{i0i})=0=
\mathrm{Hom}_{FG}(F\otimes_W X_{i0i},F\otimes_W X_{j0i}).$$
Since 
$$\mathrm{Hom}_{kG}(Y_{i0j},Y_{i0i})\cong k\cong
\mathrm{Hom}_{kG}(Y_{i0i},Y_{j0i}),$$
it follows from Lemma \ref{lem:Wlift} that both $Z_{i0j0}$ and $Z_{0j0i}$ have  a lift over $W$.
Moreover, if $(i,j)=(1,2)$ then the $F$-character of these lifts is $\chi_2+\chi_4+\chi_5$, and if
$(i,j)=(2,1)$ then the $F$-character of these lifts is $\chi_3+\chi_4+\chi_6$.
Since $\overline{U}_{i0j0}\cong \Omega^{-1}(Z_{i0j0})$ (resp. $\overline{U}_{0i0j}\cong \Omega(
Z_{0i0j})$), we obtain that  $\overline{U}_{i0j0}$ (resp. $\overline{U}_{0i0j}$) also
has a lift $(U'_{i0j0},\nu'_{i0j0})$ (resp. $(U'_{0i0j},\nu'_{0i0j})$) over $W$. Because of the $F$-characters of 
$Z_{i0j0}$ and $Z_{0i0j}$, it follows that the $F$-character of  $U'_{i0j0}$ (resp. $U'_{0i0j}$)
is
$$\sum_{\ell=2}^{d}\rho_\ell=\sum_{u=1}^{2^{d-1}-1}\chi_{7,u}.$$
If $M$ is isomorphic to $M_{i0j0}$ (resp. $M_{0i0j}$), let $\overline{U}$ be $\overline{U}_{i0j0}$
(resp. $\overline{U}_{0i0j}$), let $\phi$ be $\phi_{i0j0}$ (resp. $\phi_{0i0j}$),
let $\mu_t$ be $\mu_{t,i0j0}$ (resp. $\mu_{t,0i0j}$),
let $U'$ be $U'_{i0j0}$ (resp. $U'_{0i0j}$) and let $\nu'$ be $\nu'_{i0j0}$ (resp. $\nu'_{0i0j}$). 
By Lemma \ref{lem:quaternionargument}, $U'$ is an $R'G$-module which is 
free as an $R'$-module and $\mathrm{End}_{WG}(U')\cong R'$
where $R'=W[[t]]/(p_{d+1}(t))$. Let $\Psi'_t:U'\to U'$ be the $WG$-module endomorphism of $U'$ which
defines the action of $t\in R'$ on $U'$ and let $\overline{\Psi'_t}:U'/2U'\to U'/2U'$ be the induced
$kG$-module endomorphism of $U'/2U'$. Since $\nu':U'/2U'\to \overline{U}$ is a $kG$-module isomorphism
and $\mathrm{End}_{kG}(U'/2U')\cong R'/2R'\cong k[t]/(t^{2^{d-1}-1})$, it follows that
the $kG$-module endomorphism $\psi'_t=\nu'\circ\overline{\Psi'_t}\circ (\nu')^{-1}$ of $\overline{U}$
satisfies $\psi'_t(\overline{U})=\mu_t(\overline{U}) =t\overline{U}$. 
Thus we obtain a composition of $kG$-modules 
$k\otimes_{R'}U'=U'/(2,t) U'\xrightarrow{\overline{\nu'}} \overline{U}/t\overline{U} \xrightarrow{\phi} M$,
where $(2,t)$ denotes the ideal of $R'$ generated by $2$ and $t$ and
$\overline{\nu'}$ is the $kG$-module isomorphism induced by $\nu'$. If
$\phi'=\phi\circ\overline{\nu'}$, then
$(U',\phi')$ is a lift of $M$ over $R'=W[[t]]/(p_{d+1}(t))$.
We therefore have a 
continuous $W$-algebra homomorphism $\tau:R(G,M)\to R'$ relative to $(U',\phi')$. Since $U'/2U'\cong
\overline{U}$ is indecomposable as a $kG$-module, $\tau$ must be surjective. 
Hence $\tau$ induces a surjective $k$-algebra homomorphism $\overline{\tau}:
R(G,M)/2R(G,M)\to R'/2R'$. Since $R(G,M)/2R(G,M)\cong R'/2R'$ are both finite dimensional over $k$, 
this implies that $\overline{\tau}$ is an isomorphism. Because $R'$ is a free $W$-module of finite rank, 
it follows that $\tau$ is an isomorphism.
By Lemma \ref{lem:quaterniontrick}, $R'$ is isomorphic to a subquotient ring of $WD$.
This completes the proof of part (i), and hence of Theorem \ref{thm:altogether}.
\end{proof}

Since the case $d\ge 4$ in family (III) was excluded in \cite{3sim}, we now determine the universal
deformation rings $R(G/N,M)$ when $\mathrm{End}_{k[G/N]}(M)\cong k$ and $M$ belongs to a block
$\overline{B}$ as in \S\ref{ss:a7A} for $d\ge 4$. We prove a slightly more general result.

\begin{lemma}
\label{lem:moreIII}
Let $k$ be an algebraically closed field of characteristic $2$, let $H$ be a finite group and let $d\ge 4$
be an integer. Suppose $B_H$ is a block of $kH$ with a dihedral defect group $D_H$ of order $2^d$ 
such that $B_H$ is Morita equivalent to $\overline{\Lambda}=kQ/\overline{I}$ where $Q$ and 
$\overline{I}$ are as in \S$\ref{ss:a7A}$.
Denote the three simple $B_H$-modules by $T_0$, $T_1$ and $T_2$, where 
$T_i$ corresponds to the simple $\overline{\Lambda}$-module $S_i$, for $i\in\{0,1,2\}$,
under the Morita equivalence. 
Suppose $M$ is a $kH$-module belonging to $B_H$ with $\mathrm{End}_{kH}(M)\cong k$.
\begin{enumerate}
\item[(i)] If $M$ is not isomorphic to $T_i$ for $i\in\{1,2\}$ and $M$ is not uniserial of length $4$, 
then $R(H,M)\cong W$.
\item[(ii)] If $M$ is isomorphic to $T_2$ or $M$ is uniserial of length $4$, then $R(H,M)\cong k$.
\item[(iii)] If $M$ is isomorphic to $T_1$, then $R(H,M)\cong W[[t]]/(t\,p_d(t),2\,p_d(t))$, where
$p_d(t)\in W[t]$ is as in Definition $\ref{def:seemtoneed}$ $($when $d+1$ is replaced by $d$$)$.
\end{enumerate}
In all cases $(i)-(iii)$, $R(H,M)$ is isomorphic to a subquotient ring of $WD_H$. If $M$ is as in part $(iii)$,
then $R(H,M)$ is not a complete intersection.
\end{lemma}

\begin{proof}
Since $B_H$ is Morita equivalent to $\overline{\Lambda}$ as in \S\ref{ss:a7A} and since the 
endomorphism ring of $M$ is isomorphic to $k$, we can use the list in Remark \ref{rem:stablend}(iii) 
to describe all possibilities for $M$.
It follows that
$\mathrm{Ext}^1_{kH}(M,M)=0$ if $M$ is not isomorphic to $T_1$ and 
$\mathrm{Ext}^1_{kH}(T_1,T_1)\cong k$. 
By \cite[p. 295]{erd}, the decomposition matrix of $B_H$ has the form
\begin{equation}
\label{eq:help}
\begin{array}{ccc}
&\begin{array}{c@{}c@{}c}\varphi_0\,&\,\varphi_1\,&\,\varphi_2\end{array}\\[1ex]
\begin{array}{c}\chi_1\\ \chi_2\\ \chi_3 \\ \chi_4\\  \chi_{5,i}\end{array} &
\left[\begin{array}{ccc}1&0&0\\1&1&0\\1&0&1\\1&1&1\\0&1&0\end{array}\right]
&\begin{array}{c}\\ \\ \\ \\1\le i\le 2^{d-2}-1.\end{array}
\end{array}
\end{equation}
Moreover, the description in \cite[\S3.4]{3sim} of the ordinary characters belonging to blocks with dihedral defect 
groups and precisely three isomorphism classes of simple modules is in particular valid for $B_H$.
It follows from this description that $\chi_1,\chi_2,\chi_3,\chi_4$ are the characters of simple $FH$-modules
and that $\sum_{i=1}^{2^{d-2}-1}\chi_{5,i}$ is the character of a semisimple $FH$-module. 

Using the decomposition matrix in (\ref{eq:help}) together with \cite[Prop. (23.7)]{CR}, it
follows that the modules $M$ in parts (i) and (iii) of
Lemma \ref{lem:moreIII} have a lift over $W$ but not the modules $M$ in part (ii). In particular, this 
implies that $R(H,M)\cong W$ for all $M$ as in part (i).

If $M$ is as in part (ii), then $M$ belongs to the boundary of a $3$-tube of the stable Auslander-Reiten
quiver of $B_H$. Hence we can use the same arguments as in the proof of \cite[Cor. 5.2.5]{3sim} to
see that $R(H,M)\cong k$ for all $M$ as in part (ii).

Finally, let $M$ be as in part (iii), i.e. $M\cong T_1$. Let $R=R(H,T_1)$ and $\overline{R}=R/2R$.
Considering the projective $B_H$-module cover $P_{T_1}$ of $T_1$, it follows that $P_{T_1}$ 
has a uniserial quotient module $\overline{U}$ of length $2^{d-2}$ whose composition factors are 
all isomorphic to $T_1$ such that $\mathrm{Ext}^1_{kH}(\overline{U},T_1)=0$ and such that
there are $kH$-module isomorphisms 
$\phi:\overline{U}/\mathrm{rad} (\overline{U})\to T_1\,$, 
$\phi':\mathrm{rad}(\overline{U})/\mathrm{rad}^2 (\overline{U})\to T_1$ and
$\psi:\overline{U}/\mathrm{rad}^{2^{d-2}-1}(\overline{U}) \to \mathrm{rad}(\overline{U})$.
Hence it follows by Lemma \ref{lem:trythis} that 
$\overline{R}\cong k[t]/(t^{2^{d-2}})$
and that the universal mod $2$ deformation of $T_1$ over $\overline{R}$ is  
$[\overline{U},\phi]$, where the action of $t$ on $\overline{U}$ is 
given by the $kH$-module endomorphism $\mu_t$ of $\overline{U}$
which is induced by $\psi$. 
Let $\overline{U'}=\mathrm{rad}(\overline{U})$. Then $(\overline{U'},\phi')$ is a lift of $T_1$ over
$k[t]/(t^{2^{d-2}-1})$, where the action of $t$ on $\overline{U'}$ is given by the $kH$-module
endomorphism $\mu'_t$ of $\overline{U'}$ which is the restriction of $\mu_t$ to
$\overline{U'}$.
Using the decomposition matrix in (\ref{eq:help}) together with 
\cite[Prop. (23.7)]{CR}, we see that $\overline{U'}$ has a lift $(U',\xi')$ over $W$ such that 
the $F$-character of $U'$ is equal to $\sum_{i=1}^{2^{d-2}-1}\chi_{5,i}$. Using similar 
arguments as in the proof of \cite[Thm. 5.1]{3sim} and employing
the description of the ordinary characters belonging to $B_H$ in \cite[\S3.4]{3sim}, it follows that
$U'$ defines a lift $(U',\nu')$ of $T_1$ over $W[[t]]/(p_d(t))$, where $p_d(t)$ is as in Definition 
\ref{def:seemtoneed} (when $d+1$ is replaced by $d$). We therefore have a continuous $W$-algebra homomorphism 
$\tau':R=R(H,T_1)\to W[[t]]/(p_d(t))$ relative to $(U',\nu')$. Since $k\otimes_WU'\cong \overline{U'}$ is 
indecomposable as a $kH$-module, 
$\tau'$ must be surjective.  By \cite[Lemma 2.3.3]{3sim},  it follows that 
$R\cong W[[t]]/(p_d(t)(t-2c),a\,2^mp_d(t))$ for certain $c\in W$, $a\in\{0,1\}$ 
and $0 < m\in\mathbb{Z}$. 
Let $[U,\nu]$ be the universal deformation of $T_1$ over $R$. Since $U$ is free over
$R$ and since we can identify $k\otimes_WR=R/2R=\overline{R}$, it follows that we 
can also identify $k\otimes_W U = U/2U = \overline{R}\otimes_R U$ as $\overline{R}H$-modules. 
Hence $[U/2U,\nu]$ is
equal to the universal mod $2$ deformation $[\overline{U},\phi]$ of $T_1$ over $\overline{R}$.
In particular, there is a $kH$-module isomorphism $\xi:U/2U\to \overline{U}$.
If $a=0$, then $R\cong W[[t]]/(p_d(t)(t-2c))$ is free over $W$, which implies that
$(U,\xi)$ is a lift of $\overline{U}$ over $W$.
If $a=1$, then $(W/2^mW)\otimes_W R\cong (W/2^m W)[[t]]/(p_d(t) (t-2c))$ is free over 
$W/2^mW$, which implies that $((W/2^mW)\otimes_W U, \xi)$ is a lift of $\overline{U}$ over 
$W/2^mW$. Since $\overline{U}$ lies in the $\Omega$-orbit of a uniserial $B_H$-module of
length $4$, it follows that $R(H,\overline{U})\cong k$ by what we have proved above for the modules
in part (ii). Hence $a=1$ and $m=1$, which implies $R=R(H,T_1)\cong W[[t]]/(t\,p_d(t),2\,p_d(t))$.
\end{proof}


\begin{figure}[ht] 
\caption{\label{fig:psl1B} The projective indecomposable
modules for blocks $B$ in family (I).}
$$P_0=\vcenter{ \xymatrix @R=.1pc @C=.3pc{&0&\\
1&&2\\0\ar@{-}[rrdddddddd]&&0\ar@{-}[lldddddddd]\\2&&1\\0&&0\\:&&:\\:&&:\\0&&0\\1&&2\\0&&0\\2&&1\\&0&}},
\qquad P_1=\vcenter{ \xymatrix @R=.1pc @C=.3pc{&1&\\&0\ar@{-}[ldddd]&\\&&2\\&&0\\&&1\\
1\ar@{-}[rdddd]&&:\\&&:\\&&0\\&&2\\&0&\\&1&}}, 
\qquad P_2=\vcenter{ \xymatrix @R=.1pc @C=.3pc{&2&\\&0\ar@{-}[ldddd]&\\&&1\\&&0\\&&2\\
2\ar@{-}[rdddd]&&:\\&&:\\&&0\\&&1\\&0&\\&2&}}.$$
\end{figure}
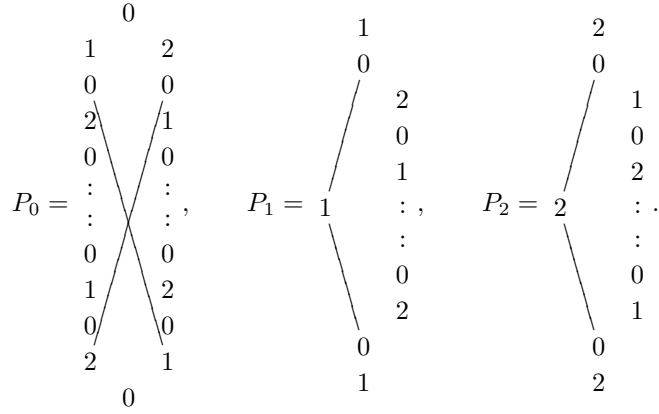
\begin{figure}[ht]  
\caption{\label{fig:decompsl1} The decomposition matrix for blocks $B$
in family (I).}
$$\begin{array}{ccc}
&\begin{array}{c@{}c@{}c}\varphi_0\,&\,\varphi_1\,&\,\varphi_2\end{array}\\[1ex]
\begin{array}{c}\chi_1\\ \chi_2\\ \chi_3 \\ \chi_4\\ \chi_5\\ \chi_6\\ \chi_{7,i}\end{array} &
\left[\begin{array}{ccc}1&0&0\\1&1&0\\1&0&1\\1&1&1\\0&1&0\\0&0&1\\2&1&1\end{array}\right]
 &\begin{array}{c}\\ \\ \\ \\ \\ \\1\le i\le 2^{d-1}-1.\end{array}\end{array}$$
\end{figure}
\begin{figure}[ht] 
\caption{\label{fig:psl2B} The projective indecomposable
modules for blocks $B$ in family (II).}
$$P_0=\vcenter{ \xymatrix @R=.1pc @C=.2pc{&0&\\
1\ar@{-}[rrdd]&&2\ar@{-}[lldd]\\0&&0\\1&&2\\&0&}},
\qquad P_1=\vcenter{ \xymatrix @R=.1pc @C=.2pc{&1&\\0\ar@{-}[rrdddddd]&&2\ar@{-}[lldd]\\
1&&1\\0&&2\\&&:\\&&:\\&&1\\&&2\\&1&}}, \qquad 
P_2=\vcenter{ \xymatrix @R=.1pc @C=.2pc{&2&\\0\ar@{-}[rrdddddd]&&1\ar@{-}[lldd]\\
2&&2\\0&&1\\&&:\\&&:\\&&2\\&&1\\&2&}}.$$
\end{figure}
\begin{figure}[ht]  
\caption{\label{fig:decompsl2}The decomposition matrix for blocks  $B$ in family (II).}
$$\begin{array}{ccc}
&\begin{array}{c@{}c@{}c}\varphi_0\,&\,\varphi_1\,&\,\varphi_2\end{array}\\[1ex]
\begin{array}{c}\chi_1\\ \chi_2\\ \chi_3 \\ \chi_4\\ \chi_5\\ \chi_6\\ \chi_{7,i}\end{array} &
\left[\begin{array}{ccc}1&0&0\\0&1&0\\0&0&1\\1&1&1\\1&1&0\\1&0&1\\0&1&1\end{array}\right]
&\begin{array}{c}\\ \\ \\ \\ \\ \\1\le i\le 2^{d-1}-1.\end{array}
\end{array}$$
\end{figure}
\begin{figure}[ht] 
\caption{\label{fig:a7B} The projective indecomposable modules
for blocks $B$ in family (III).}
$$P_0=\vcenter{\xymatrix @R=.1pc @C=.2pc {&0&\\
1\ar@{-}[rrdddddd]&&2\\0&&0\ar@{-}[llddddd]\\2&&1\\ 0&&0\\1&&2\\0&&0\\
2&&1\\&0&}},
\qquad P_1=\vcenter{\xymatrix @R=.1pc @C=.2pc {&1&\\1\ar@{-}[rrdddddd]&&0\ar@{-}[lldddddd]\\ 
1&&2\\:&&0\\ :&&1\\ :&&0\\ 1&&2\\1&&0\\&1&}},
\qquad P_2=\vcenter{\xymatrix @R=.1pc @C=.2pc {&2&\\&0\ar@{-}[lddd]&\\
&&1\\&&0\\2\ar@{-}[rddd]&&2\\&&0\\&&1\\&0&\\&2&}}$$
\vspace{1ex}
\end{figure}
\begin{figure}[ht]  
\caption{\label{fig:decoma7}The decomposition matrix for blocks  
$B$ in family (III).}
$$\begin{array}{ccc}
&\begin{array}{c@{}c@{}c}\varphi_0\,&\,\varphi_1\,&\,\varphi_2\end{array}\\[1ex]
\begin{array}{c}\chi_1\\ \chi_2\\ \chi_3 \\ \chi_4\\ \chi_5\\ \chi_6\\ \chi_{7,i}\end{array} &
\left[\begin{array}{ccc}1&0&0\\1&1&0\\1&0&1\\1&1&1\\2&1&1\\0&0&1\\0&1&0\end{array}\right]
&\begin{array}{c}\\ \\ \\ \\ \\ \\1\le i\le 2^{d-1}-1.\end{array}
\end{array}$$
\vspace{1ex}
\end{figure}

\end{document}